\documentclass[a4paper,american,11pt, reqno]{amsart}

\makeatletter
\def\blfootnote{\xdef\@thefnmark{}\@footnotetext}
\makeatother
\usepackage[utf8]{inputenc} 
\usepackage[T1]{fontenc} 
\usepackage{mathtools}
\usepackage{amsmath}
\usepackage{amssymb}
\usepackage{hyperref}
\usepackage{cleveref}
\usepackage{url}
\usepackage{babel}
\usepackage{pgf,tikz,pgfplots}
\usepackage{graphicx}
\usepackage{tikz-cd}
\usepackage{tikz-3dplot}
\usepackage{csquotes} 

\usepackage{tikz, float} \usetikzlibrary {positioning}
\usetikzlibrary{shapes.misc}
\usetikzlibrary{arrows}

\usetikzlibrary{calc}
\usepackage{amsfonts}
\input xy
\xyoption{all}

\newcommand{\Z}{{\mathbb Z}}
\newcommand{\NN}{{\mathbb N}}
\newcommand{\PP}{{\mathbb P}}
\newcommand{\KK}{{\mathbb K}}
\newcommand{\C}{{\mathbb C}}
\newcommand{\RR}{{\mathbb R}}
\newcommand{\R}{{\mathbb R}}
\newcommand{\QQ}{{\mathbb Q}}

\newcommand{\T}{{\mathbb T}}
\newcommand{\TP}{{\T\PP}}
\newcommand{\RP}{{\RR\PP}}
\newcommand{\GL}{{\text{GL}}}
\newcommand{\CP}{{\C\PP}}

\newcommand{\affFan}{\Sigma}
\newcommand{\projFan}{\PP\Sigma}

\newcommand{\EEE}{{\mathcal E}}

\newcommand{\CCC}{{\mathcal C}}
\newcommand{\FFF}{{\mathcal F}}

\newcommand{\MMM}{{\mathcal M}}

\newcommand{\OOO}{{\mathcal O}}
\newcommand{\PPP}{{\mathcal P}}
\newcommand{\XXX}{{\mathcal X}}
\newcommand{\YYY}{{\mathcal Y}}

\newcommand{\XX}{{\mathbf X}}

\newcommand{\Tbold}{{\mathbf T}}

\DeclareMathOperator{\cl}{cl}
\DeclareMathOperator{\Trop}{Trop}
\DeclareMathOperator{\sign}{sign}
\DeclareMathOperator{\Int}{Int}
\DeclareMathOperator{\BC}{BC}
\DeclareMathOperator{\Pow}{Pow}

\DeclareMathOperator{\Spec}{Spec}

\newcommand{\PW}{\mathfrak{R}}

\newcommand{\D}{{\mathcal D}}
\newcommand{\E}{{\mathcal E}}
\newcommand{\F}{{\mathcal F}}

\newcommand{\X}{{\mathcal X}}
\newcommand{\G}{{\mathcal G}}
\newcommand{\Y}{{\mathcal X}'}
\newcommand{\M}{{\mathcal M}}

\renewcommand{\S}{{\mathcal S}}
\renewcommand{\P}{{\mathcal P}}
\renewcommand{\M}{{\mathcal M}}
\renewcommand{\L}{{\mathcal L}}

\newcommand{\bL}{\textbf{L}}

\newcommand{\val}{\text{val}}

\newcommand{\Aff}{\text{Aff}}

\newcommand{\BM}{\text{BM}}

\newcommand{\K}{{\mathcal K}}
\newcommand{\Rnn}{\R_\geq}
\newcommand{\Rpos}{\R_>}
\newcommand{\PWpos}{\PW_>}
\newcommand{\PWnn}{\PW_\geq}

\DeclareMathOperator{\Star}{star}
\DeclareMathOperator{\pStar}{Pstar}

\DeclareMathOperator{\relint}{relint}

\DeclareMathOperator{\sed}{sed}
\DeclareMathOperator{\RecCone}{RecCone}

\DeclareMathOperator{\Bl}{Bl}
\DeclareMathOperator{\Conv}{Conv}
\DeclareMathOperator{\init}{in}

\DeclareMathOperator{\Facets}{Facets}

\DeclareMathOperator{\Bounded}{Bounded}

\DeclareMathOperator{\Log}{Log}
\DeclareMathOperator{\Hom}{Hom}

\newtheorem{thm}{Theorem}[section]
\newtheorem{defi}[thm]{Definition}

\newtheorem{definition}[thm]{Definition}
\newtheorem{prop}[thm]{Proposition}
\newtheorem{proposition}[thm]{Proposition}
\newtheorem{lemma}[thm]{Lemma}
\newtheorem{cor}[thm]{Corollary}
\newtheorem{remark}[thm]{Remark}

\newtheorem{corollary}[thm]{Corollary}

          {\theoremstyle{definition}
}
          {\theoremstyle{definition}

\newtheorem{example}[thm]{Example}
}

 \numberwithin{equation}{section}

\tikzset{%
  add/.style args={#1 and #2}{to path={%
 ($(\tikztostart)!-#1!(\tikztotarget)$)--($(\tikztotarget)!-#2!(\tikztostart)$)%
  \tikztonodes}}
} 

\newcommand{\comment}[1]{}

\begin{document}
\title[Patchworking in higher codimension]{Real phase structures on tropical varieties and 
patchworks in higher codimension} 

\author[Johannes Rau]{Johannes Rau}
\address{Johannes Rau, Departamento de Matem\'aticas, Universidad de los Andes, 
Carrera 1 \# 18A - 12, 111711 Bogot\'a, Colombia}
\email{j.rau@uniandes.edu.co}
\author[Arthur Renaudineau]{Arthur Renaudineau}
\address{Arthur Renaudineau, Univ. Lille, CNRS, UMR 8524 -
Laboratoire Paul Painlev\'e, F-59000 Lille, France.}
\email{arthur.renaudineau@univ-lille.fr}
\author[Kris Shaw]{Kris Shaw}
\address{Kris Shaw, 
University of Oslo, Oslo, Norway.}
\email{krisshaw@math.uio.no}


\begin{abstract}

This paper generalises the homeomorphism theorem behind Viro's combinatorial patchworking of hypersurfaces in toric varieties to arbitrary codimension using tropical geometry. 
We first define the patchwork of a polyhedral space equipped with a real phase structure. When the polyhedral subspace is tropically non-singular, we show that the patchwork is a topological manifold. When a non-singular tropical variety appears as a tropical limit of a real analytic family, we show that the real part of a fibre
of the family near the tropical limit is homeomorphic to the patchwork.

Finally we extend the spectral sequence introduced by  the last two authors in the case of hypersurfaces  to non-singular tropical varieties with real phase structures. As a corollary, we obtain  bounds on the Betti numbers of the patchwork in terms of the dimensions of the tropical homology groups with coefficients modulo two.

\end{abstract}
\maketitle

\tableofcontents{}

\section{Introduction}
Viro's patchworking method is a powerful method to study the topological types 
of real algebraic hypersurfaces in toric varieties \cite{Viro6and7}. 
A particular case of the method is \emph{combinatorial} patchworking which, 
from some combinatorial data, produces two objects: a polyhedral space and a patchworking polynomial. 
Viro's theorem  provides a homeomorphism between  the polyhedral space and the real part of a hypersurface defined by the polynomial. 
It was later generalized to complete intersections, see \cite{SturmfelsCompInt} and \cite{Bihan}.
This paper aims to generalise the homeomorphism theorem of combinatorial  patchworking beyond the case of hypersurfaces and complete intersections.

To  generalise the theorem to arbitrary codimension, 
we begin by considering rational polyhedral spaces and tropical varieties equipped with real phase structures, which were defined in the case of polyhedral fans in \cite{RRSmat}. 
For a rational polyhedral subspace $X$ in $\R^n$, a real phase structure is an assignment of an affine subspace of the vector space $\Z_2^n$, where $\Z_2 := \Z /2\Z$
to each facet of $X$.

The assignment of a real phase structure is then used to describe the \emph{patchwork} of a tropical variety, which should be thought of as its real part, see Definition \ref{def:realpart}.
 Our first result is a  description of  the structure of the patchwork 
when we start from a non-singular tropical subvariety.

\begin{thm}\label{thm:PLman}
The patchwork of a non-singular tropical subvariety equipped with a real phase structure in a tropical toric variety is a topological manifold. 
\end{thm}

We then show that a patchwork describes,   up to homeomorphism, fibers of  real analytic families with non-singular tropical limits. Let us describe our setup: 
Let $\XX \subset (\C^*)^n  \times \D^*$ be a real analytic family of 
algebraic varieties over the punctured disc $\D^*$.
That is, locally $\XX$ can be described by polynomials whose coefficients are Laurent series convergent on $\D^*$.
Moreover, the family is compatible with complex conjugations on $(\C^*)^n$ and $\D^*$. 
To such a family, we can associate a tropical limit $\Trop(\XX)= X \subset \R^n$, by taking the image of $\XX$ under the valuation map, see Section \ref{subsection NSTL}. 
Let us assume that $\Trop(\XX)$ is tropically non-singular, see Section \ref{subsection NSTL} for definition. Then additionally the family $\XX$ induces a real phase structure 
$\EEE$ on $\Trop(\XX)$, see Definition \ref{def:realphaselinear}.
We relate this to pre-existing notions of tropicalising real varieties in Remark \ref{sec:comparison}.

For any pointed unimodular fan $\Sigma$, we can consider the compactification of $\XX \subset (\C^*)^n \times \D^*$ in $\C \Sigma \times \D^*$. Consider a subdivision $\X$ of $X$ and the set $\RecCone(\X)$ of all recession cones for polyhedron in $\X$, see Section \ref{subsectionPolyhedralComplexes}. 
The second main theorem is that, under the assumption that $\RecCone(\X) \cup \Sigma$ forms a fan,
the patchwork $\PW(\overline{X}, \EEE(\XX))$
of $X$ (see Definition \ref{def:realpart}) in the real tropical toric variety $ \PW\T\Sigma$ of $\Sigma$ (see Equation \ref{eq:tropicalrealtoric}) gives the homeomorphism type of $\R \overline{\XX}_t \subset \R \Sigma$ for sufficiently small positive $t$.

\begin{thm}[Patchworking for non-singular tropical limits] \label{Patchworking}
	Let $\XX \subset (\C^*)^n \times \D^*$ be a real analytic family 
	with  non-singular tropical limit $X = \Trop(\XX)$ 
	and associated real phase structure $\EEE = \EEE(\XX)$. 
	Let $\X$ be a subdivision of $X$ and 
	$\Sigma$ a pointed unimodular fan such that $\RecCone(\X) \cup \Sigma$ is a fan. 

	Then for sufficiently small positive $t \in \D^* \cap \R$ the pairs 
	$(\R\Sigma, \R \overline{\XX}_t)$ and 
	$( \PW\T\Sigma, \PW(\overline{X}, \EEE))$ are homeomorphic.
	Moreover, the homeomorphism can be chosen to respect
	the stratification of $\R\Sigma$ and $\PW\T\Sigma$ by torus orbits. 
\end{thm}

We expect that the assumption that $\RecCone(\X) \cup \Sigma$ forms a fan 
can be weakened. 

We give a rough outline of the proof strategy here. The first step is the reduction to a  variant of the above theorem in which we assume the existence of a unimodular subdivision $\PPP$ of $\R^n$   compatible with $\Trop(\XX)$ in  \Cref{UnimodularPatchworking}. 
Under the unimodular assumption, the strategy is then to view both topological pairs from  \Cref{Patchworking} as regular CW pairs in the sense of \Cref{def:regpair }. In order to do this, we must prove a sequence of lemmas in \Cref{sec:local}  establishing topological properties of closures of real  linear spaces and their patchworks in toric varieties. 
We then use these lemmas and  \cite{Rau-RealSemiStable} to give a description of  $(\R\Sigma, \R \overline{\XX}_t)$ as a CW pair for small and positive $t$, see \Cref{thm:ClassicalSide}.

We then aim to compare the CW description of  $(\R\Sigma, \R \overline{\XX}_t)$ to the patchwork by introducing some intermediate CW pairs. Figure \ref{fig:proofCWcomplexes}  shows the spaces involved in the steps we now outline and the relations between them. First we introduce the notion of tropical semi-stable degeneration and of tropical special fibre arising from the subdivision $\PPP$, see \Cref{def:tropicalspecialfibre}.  
We consider the pair of tropical special fibers, denoted by $(\T \PPP_{\infty},\overline{X}_\infty)$. We prove that this pair is homeomorphic to the pair $(\T \Sigma,\overline{X})$, by using an adaptation of cubical subdivisions which we call the \emph{bounded cubical subdivision}, see \Cref{prop:TropicalSpecialFibre}.  The operation of taking this subdivision is denoted $BC$ in \Cref{fig:proofCWcomplexes}.
 We then pass to the patchwork of the  tropical positive  special fiber  $(\PW \T \PPP^+_\infty, \PW(\overline{X},\EEE)^+_\infty)$, which, it  follows, is homeomorphic to the pair $( \PW\T\Sigma, \PW(\overline{X}, \EEE))$.
  
  \begin{figure}[b]
  \begin{tikzcd}
 (\T \PPP_{\infty},\overline{X}_\infty)    \arrow[r, dotted, "patch"]                     & (\PW \T \PPP^+_\infty, \PW(\overline{X},\EEE)^+_\infty)   &  \\
 (\T\Sigma ,\overline{X} )                 \arrow[ "BC"]{1-1} \arrow[r, dotted, "patch"]  & ( \PW\T\Sigma, \PW(\overline{X}, \EEE))  \arrow[ "BC"]{1-2} & (\R\Sigma, \R \overline{\XX}_t) \arrow[ul, "cub"]
\end{tikzcd}
\caption{The CW complexes involved in the proof of  \Cref{Patchworking} 
}\label{fig:proofCWcomplexes} 
\end{figure}

We finish our  comparison of $( \PW\T\Sigma, \PW(\overline{X}, \EEE))$ and  $(\R\Sigma, \R \overline{\XX}_t)$ by showing that the  
tropical positive special fiber $(\PW \T \PPP^+_\infty, \PW(\overline{X},\EEE)^+_\infty)$ is also a cubical subdivision (denoted $cub$ in \Cref{fig:proofCWcomplexes}) of the CW pair structure that we define on $(\R\Sigma, \R \overline{\XX}_t)$  in \Cref{thm:ClassicalSide}.
Lastly, we use that CW pairs with isomorphic face posets are homeomorphic. 

We say that a real algebraic variety is \emph{close to a non-singular limit}
if it is isomorphic to a fibre $\overline{\XX}_t$ for which the assumptions of \Cref{Patchworking}
hold. 
The homeomorphism from \Cref{Patchworking}  implies that statements about the topology of $\PW(\overline{X}, \EEE)$ translate 
into statements about the topology of real algebraic varieties close to a non-singular  tropical limit.

Next, we adapt 
the arguments from \cite{RS} to bound the Betti numbers of the patchwork in terms of the dimensions of the  tropical homology groups. 
For a tropical manifold $X$, we denote the tropical homology groups with $\Z_2$ coefficients by 
$H_q(X; \F_p)$. We let $H_q^{BM}(X; \F_p)$ denote the Borel-Moore variants of these homology groups. 
See Section \ref{sec:trophom} for the definition of tropical homology.

We define the (Borel-Moore) tropical  signature of $X$ as 
\[
  \sigma^\diamond(X) := \sum_{p,q} (-1)^q H_{q}^\diamond(X; \F_p), 
\]
where $\diamond$ is either empty, denoting usual homology, or $\diamond$ is BM, denoting Borel-Moore homology. 
Notice that $\sigma^\diamond(X)$ is also equal to $\sum_{p,q} (-1)^q H_{q}^\diamond(X; \F_p^{\R})$ 
where $H_{q}^\diamond(X; \F_p^{\R})$ denotes tropical homology with real coefficients.

\begin{thm}\label{thm:boundbetti}
If  $X$ is a $d$-dimensional non-singular tropical subvariety equipped with a real phase structure $\E$, then 
$$b_q (\PW(X, \EEE)) \leq \sum_{p = 0}^d \dim H_q(X; \F_p) $$ and
$$ b^{BM}_q (\PW(X, \EEE)) \leq \sum_{p = 0}^d \dim H^{BM}_q(X; \F_p).$$

Moreover, the (Borel-Moore) Euler characteristic of $\PW (X, \EEE)$ is equal to the 
	(Borel-Moore) tropical signature of $X$, namely, 
	\begin{align*} 
    \chi(\PW(\X,\EEE)) = \sigma(X), && \chi^\text{BM}(\PW(\X,\EEE)) = \sigma^\text{BM}(X).
	\end{align*}
\end{thm}

We then combine Theorems \ref{Patchworking} and \ref{thm:boundbetti} to obtain the following two corollaries. 

\begin{corollary}\label{cor:boundsSemiStable}
Under the same hypothesis as in Theorem \ref{Patchworking}, for sufficiently small $t \in \D^* \cap \R$ we have 
\[
  b_q (\RR \overline{\XX}_t ) \leq \sum_{p = 0}^d \dim H_q(\overline{X}; \F_p) \quad \text{ and } \quad b^{BM}_q (\RR \overline{\XX}_t) \leq \sum_{p = 0}^d \dim H^{BM}_q(\overline{X}; \F_p).
\]
\end{corollary}

We define the \emph{signature} of a generic fibre $\overline{\XX}_t$ as 
\[
  \sigma_c(\overline{\XX}_t) := \sum_{p,q} (-1)^p e_c^{p, q}(\overline{\XX}_t)
\]
where $e_c^{p, q}(\overline{\XX}_t)$ is defined via the 
mixed Hodge structure on $\overline{\XX}_t$ by 
$ e_c^{p, q}(\overline{\XX}_t):= \sum_k (-1)^k h^{p,q}(H_c^k (\overline{\XX}_t)).$ 
See for instance \cite{DanilovKhovansky}.

\begin{corollary}\label{cor:eulerSignature}
Let $\XX \subset \C\Sigma \times \D^*$ be a real meromorphic family 
with non-singular tropical limit satisfying the same hypothesis as in Theorem \ref{Patchworking}.
Then for sufficiently small $t \in \D^* \cap \R$ we have 
\[
  \chi^\BM(\RR \overline{\XX}_t )  = \sigma_c(\overline{\XX}_t). 
\]
\end{corollary}

This equality was already proved in the case of codimension one by Bertrand \cite{Bertrand10} (see also \cite{Arnal}), and  by Brugallé in the same situation as above  by utilising a  different point of view \cite{brugalle2021euler}.

In a recent paper \cite{ambrosimanzaroli}, the authors generalized the ideas of \cite{RS} beyond the case of tropical varieties. They proved that the same bounds still apply for any totally real semi-stable degeneration satisfying strong homological properties. Essentially what they require is that the open strata of the special fiber are all maximal in the sense of the Smith-Thom inequality, the real parts only have non-trivial homology in degree $0$, and that the cohomology of the complex open strata satisfy a strong condition on their mixed Hodge structures.  These properties are satisfied in our context and so, up to a  comparison of tropical homology and the complexes obtained by Ambrosi and Manzaroli, this eventually gives another proof of \Cref{cor:boundsSemiStable}.

We would also like to mention the work of Brugallé, López de Medrano and the first author \cite{BLdMR-CombinatorialPatchworkingBack}.
This work transfers tropical homology, real phase structures, and the Betti number bounds of \Cref{thm:boundbetti} 
to the original framework of Viro's patchworking method, namely the framework of (unimodular) triangulations
of lattice polytopes. An overlap with the present work is given if the triangulation is assumed to be convex. Then the objects of study 
in \cite{BLdMR-CombinatorialPatchworkingBack} correspond to self-intersections of a
tropical hypersurface associated to the triangulation (and the correspondence can supposedly extended to complete intersections
using mixed subdivisions). 
So while \cite{BLdMR-CombinatorialPatchworkingBack} also covers non-convex triangulations which do not
have a tropical counterpart, the present approach also treats non-complete intersections
which seem to have no counterpart on the triangulation side. 

\section*{Acknowledgements}

The authors are grateful to Erwan Brugallé, Ilia Itenberg, Oleg Viro for helpful conversations and comments. 

This research was supported in part by the Trond Mohn Foundation project ``Algebraic and Topological Cycles in Complex and Tropical Geometries''. 
 We also acknowledge the support of the Centre for Advanced Study (CAS) in Oslo, Norway, which funded and hosted the Young CAS research project Real Structures in Discrete, Algebraic, Symplectic, and Tropical Geometries (REACTIONS) during the 2021-23 academic years.  

Johannes Rau was supported by the FAPA project ``Matroids in tropical geometry'' from the Facultad
de Ciencias, Universidad de los Andes, Colombia.

Arthur Renaudineau acknowledges support from the Labex CEMPI (ANR-11-LABX-0007-01), and from ANR, project ANR-22-CE40-0014.

\section{Real phase structures} 

\label{sec:realtropspaces}

\subsection{Tropical toric varieties}
Given a polyhedron $\sigma \subset \R^n$ we denote by $T(\sigma)$ its
tangent space, that is, the vector subspace of $\R^n$ generated by
all vectors $u-v$, $u,v \in \sigma$. We are often interested 
in the quotient space $\R^n / T(\sigma)$ for which we hence use the shorthand 
$\R(\sigma)$.
If $\sigma$ is a rational polyhedron, i-e the supporting hyperplanes are defined over $\QQ$,
we denote by $T_\Z(\sigma) = T(\sigma) \cap \Z^n$ its \emph{tangent lattice},
and set $\Z(\sigma) = \Z^n / T_{\Z}(\sigma)$. 
Finally, for a $\Z$-module $G$, we set $T_G(\sigma) = T_\Z(\sigma) \otimes G$
and $G(\sigma) = G^n / T_G(\sigma)$.

Let $\Sigma$ be a rational polyhedral pointed fan in $\R^n$.
Given $\rho \in \Sigma$, we denote the dual cone by
\[
  \rho^\vee = \{\lambda \in (\R^n)^\vee : \lambda|_\rho \geq 0 \} \subset (\R^n)^\vee.
\]
We denote by $\T U_\rho = \Hom(\rho^\vee \cap (\Z^n)^\vee, \T)$ the space
of semigroup homomorphisms, where $\T=\RR\cup\left\lbrace -\infty \right\rbrace$ endowed with the addition is the \emph{tropical} semigroup.  We equip it with the coarsest topology 
such that for all $\lambda \in \rho^\vee \cap (\Z^n)^\vee$ 
the evaluation map $\T U_\rho \to \T$, $\alpha \mapsto \alpha(\lambda)$, 
is continuous. If $\rho \subset \eta$, we 
have by restriction a continuous map $\T U_\rho \to \T U_\eta$. 

\begin{definition}\cite[Chapter 3]{MikRau} \cite[Section 3]{PayneAnalytification}
  The tropical toric variety associated to the fan $\Sigma$ 
	is the direct limit
	\[
	 \T \Sigma := \varinjlim_{\rho \in \Sigma} \T U_\rho.
	\]
	We call $\T \Sigma$ \emph{non-singular} 
	if $\Sigma$ is unimodular (and hence simplicial).
\end{definition}

The space $\T \Sigma$ is compact if and only if the corresponding fan is complete. 
The action of $\R^n$ on itself extends to a group action on $\T \Sigma$
whose orbits are given by $\T \OOO_\rho = \R(\rho) \cong \Hom(\rho^\perp, \T)
\subset \T U_\rho$ (the inclusion is given by extension by $-\infty$).
Notice that $\dim \T \OOO_{\rho} = n- \dim \rho$.
Given a point $y \in \T \OOO_\rho$, we call $\rho$ the \emph{sedentarity} of $y$. 
The \emph{order of sedentarity} of $y$, denoted by $\sed(y)$, is set to be $\dim(\rho)$.

\begin{example} \label{exAffineSpace}
  Tropical affine space $\T^n$ is the tropical toric variety associated 
	to the fan that consists of the negative orthant in $\R^n$ 
	and all its faces. 
	Given a subset $A$ of $E = \{1, \dots, n\}$, 
	we denote by $\rho_A$ the cone generated by $\{-e_i, i \in A\}$, where the $e_i$ are the vectors of the standard basis of $\R^n$.
	Then the torus orbit
	$\R^n_A := \T\OOO_{\rho_A}$ is equal to $\R^{E \setminus A}$, while its  closure in $\T^n$, denoted by $\T^n_A$, 
	can be identified with $\T^{E \setminus A}$. \\
	Similarly, we set $e_0 = -e_1 - \dots - e_n$
	and for any \emph{proper} subset  $A \subset \{0, \dots, n\}$
	we let $\sigma_A$ be the cone generated by $\{-e_i, i \in A\}$.
	The collection of these cones forms a complete fan defining
	the tropical toric variety $\TP^n$. 
	We denote the closure of $\T\OOO_{\rho_A}$ in $\TP^n$ by 
	$\TP^n_A$. 
	It can be identified with a projective space $\TP^{n-|A|}$
	(whose homogeneous coordinates are labelled by $E \setminus A$). 
\end{example}

\begin{definition}
  Given two fans $\Sigma$ in $\R^n$ and $\Sigma'$ in $\R^m$, 
	a \emph{morphism of fans} 
	$A \colon \Sigma \to \Sigma'$ is a $\Z$-linear map $A \colon \R^n \to \R^m$
	such that $A(\sigma) \in \Sigma'$ for all $\rho \in \Sigma$. 
	
  Let $\T \Sigma$ and $\T \Sigma'$ denote the tropical toric varieties associated 
	to $\Sigma$ and $\Sigma'$. 
	A \emph{(toric) morphism} $\psi \colon \T \Sigma \to \T \Sigma'$  is 
	a continuous map that sends $\R^n \subset \T \Sigma$ to 
	$\R^m \subset \T \Sigma'$ and such that the restriction
	$\R^n \to \R^m$ is of the form $x \mapsto Ax + v$ where
	$A$ is a morphism of fans from $\Sigma$ to $\Sigma'$.
\end{definition}

We note that given $x \mapsto Ax + v$, there is a unique continuous extension
$\psi \colon \T \Sigma \to \T \Sigma'$. Moreover, this extension sends
the torus orbit $\T\OOO_\rho$ to the torus orbit $\T\OOO_{A(\rho)}$ 
and the restriction of $\psi$ to these orbits is given by the induced map
\[
  \R(\rho) \to \R(A(\rho)), [x] \mapsto [Ax +v].
\]

\begin{example} 
  Let $\T \Sigma$ be a tropical toric variety.	The open subsets $\T U_\rho$
	are toric varieties themselves (their defining fan consists of $\rho$ and all its faces). 
	If $\rho$ is a unimodular cone generated by the first $k$ vectors
  of the $\Z^n$-basis $v_1, \dots, v_n$, then 
	$$\rho^\vee \cap (\Z^n)^\vee = 
	\langle v_1^*, \dots, v_k^*, \pm v_{k+1}^*,  \dots, \pm v_{n}^*\rangle
	\cong \NN^k \times \Z^{n-k}$$ as a cone, and hence
	$\T U_\rho \cong \T^k \times \R^{n-k}$. 
  Hence if $\T \Sigma$ is non-singular, it is covered
	by toric open subsets isomorphic to $\T^k \times \R^{n-k}$
	(for varying $k$). 
\end{example}

\begin{example} 
  The only automorphisms of $\T^n$ are coordinate permutations 
	composed with translations, so $\text{Aut}(\T^n) = \text{Sym}_n \ltimes \R^n$.
	This follows from the fact that (fan) automorphisms of the 
	negative orthant of $\R^n$ must permute its rays.
\end{example}

\begin{remark} \label{rem:Semigroups}
  We focused here on tropical toric varieties, but 
	the given construction of  toric varieties makes sense
	for any semigroup $S$
	(the given description of torus orbits requires additionally that $S$ contains 
	a single non-invertible element).
	In particular, we will later use the cases $S = \C, \R, \Rnn = [0, \infty)$
	(all with respect to multiplication) and (briefly) a few more. 
	The toric variety of the fan $\Sigma$ over the semigroup $S$ will be denoted by $S \Sigma$. 
\end{remark}

\subsection{Polyhedral complexes} \label{subsectionPolyhedralComplexes}

A rational polyhedron in a tropical toric variety $\T \Sigma$ is the 
closure in $\T \Sigma$ of a rational polyhedron in some stratum $\T \OOO_{\rho}$. 
For a polyhedron $\sigma$ in $\T \Sigma$, the intersection $\sigma \cap \T \OOO_\eta$,
if non-empty, is a polyhedron in $\T \OOO_\eta$. A \emph{face} $\tau$ of $\sigma$
is the closure of a face of any of these intersections (hence $\tau$ is a polyhedron in $\T \Sigma$).
The \emph{relative interior} $\relint(\sigma)$ is the set of points in $\sigma$ not contained in
a proper face. We set $\sed(\sigma) := \sed(y)$ for an arbitrary $y \in \relint(\sigma)$.
Therefore, we have  $\sed(\sigma) = \dim(\rho)$ referring to $\rho$ from above. 

A \emph{rational polyhedral complex} $\mathcal{X}$ in $\T \Sigma$ is a collection of rational polyhedra in $\T \Sigma$ such that 
\begin{itemize}
	\item if $\sigma \in \X$ and $\tau \subset \sigma$ is a face, then $\tau \in \X$,
	\item if $\sigma_1, \sigma_2 \in \X$ and $\sigma_1 \cap \sigma_2 \neq \emptyset$, then $\sigma_1 \cap \sigma_2$ is a face of $\sigma_1$ and $\sigma_2$. 
\end{itemize}
We refer to the elements $\sigma \in \X$ as \emph{faces} of $\X$. 
The maximal faces of $\X$ are called \emph{facets}. 
The \emph{support} $|\X|$ of $\X$ is the union of all the faces of $\X$. 
Given a polyhedral complex $\X$, a \emph{subdivision} of $\X$ is 
a polyhedral complex $\X'$ with the same support
and such that any face of $\X'$ is contained in some face of $\X$.
A \emph{rational polyhedral subspace} $X$ of $\T \Sigma$ is a 
subspace which is the support of a rational polyhedral complex $\mathcal{X}$.
Note that two polyhedral complexes have the same 
support if and only if they have a common subdivision.
Hence, alternatively we can think of a polyhedral 
subspace as an equivalence class of polyhedral complexes
under the equivalence relation induced by subdivision. 
We allow ourselves to call 
$\mathcal{X}$ a subdivision of the polyhedral space $|\X|$.

A rational polyhedral subspace $X \subset \T\Sigma$ is called 
of \emph{pure sedentarity $\rho \in \Sigma$} 
if $X = \overline{X \cap \T\OOO_\rho}$. 
A polyhedral complex is of \emph{pure sedentarity $\rho \in \Sigma$} 
if $|\X|$ is of pure sedentarity $\rho$. 
A polyhedral complex $\X$ is \emph{of pure dimension $d$} if all facets have dimension $d$. 
Clearly, the notion is preserved under subdivision and hence 
extends to polyhedral spaces $X$.
  In the following, polyhedral complexes $\X$ and polyhedral spaces $X$ 
	are \emph{always} assumed to be rational, of pure dimension,
	and of pure sedentarity. 

Let $\X$ be a polyhedral complex in $\R^n$ and $\tau \in \X$. 
By $\Star_\tau \X$ we denote the fan in $\R(\tau)$
whose cones are given by
\begin{align*} 
  \frac { \sigma + T(\tau)}{T(\tau)}, && \tau \subseteq \sigma \in \X.  
\end{align*}
  If $\X$ is a polyhedral complex in a tropical toric variety and $\tau \in \X$
of sedentarity $\rho$, we set 
$\Star_\tau \X = \Star_{\tau \cap \T \OOO_\rho}(\X \cap \T \OOO_\rho)$. 
We note that $\T \Star_\rho \Sigma$ is naturally embedded in 
$\T\Sigma$ as the closure of the 
torus orbit $\T \OOO_\rho$.

Given a polyhedron $\sigma \subset \R^n$, its \emph{recession cone}
is defined by 
\[
  \RecCone(\sigma) = \{v \in \R^n : \lambda v + p \in \sigma 
	\;\forall\; p \in \sigma, \lambda \in \Rnn \}.
\]
If $\sigma \subset \T \Sigma$ is a polyhedron of sedentarity $\rho$, we set
$\RecCone(\sigma) = \RecCone(\sigma \cap \T \OOO_\rho)$. 
Given a polyhedral complex $\X$, we set $\RecCone(\X) = \{\RecCone(\sigma) : \sigma \in \X\}$.

Let us assume that $\X$ is a polyhedral complex in $\R^n$ such that 
$\Sigma = \RecCone(\X)$ is a fan in $\R^n$. Let us additionally assume that
$\X$ contains a vertex, or equivalently, that $\Sigma$ is pointed. 
We will now define a polyhedral complex $\overline{\X}$
in $\T\Sigma$ with the property that $|\overline{\X}| = \overline{|\X|}$
and $\overline{\X} \cap \R^n = \X$. 
The construction is sometimes called the canonical compactification of
$\X$, cf.\ \cite{KSW}, \cite[Section 2.4]{AP}. 
We set 
\begin{equation} \label{eq:CanonicalComp} 
  \overline{\X} := 
	  \{ \overline{\sigma} \cap \T\Star_\rho \Sigma : \sigma \in \X, \rho \in \Sigma\}.
\end{equation}
Note that, for example by \cite[Proposition 3.2.4]{MikRau},
the intersection $\overline{\sigma} \cap \T\Star_\rho \Sigma$ is non-emtpy
if and only if $\rho \subset \RecCone(\sigma)$,
and moreover, in this case 
$\overline{\sigma} \cap \T\Star_\rho \Sigma$ is the closure in $\T\Star_\rho \Sigma$
of 
\[
  \sigma/T(\rho) \subset \R(\rho) = \T\OOO_\rho.
\]
For further reference, we note that the labelling of faces of $\overline{\X}$ 
by pairs $(\sigma, \rho)$ provides 
an isomorphism between the poset $\overline{\X}$ 
and the subposet of $\X \times \Sigma$ 
given by
\begin{equation} \label{eq:PosetCompactification}
  \{(\sigma, \rho) \in \X \times \Sigma : \rho \subset \RecCone(\sigma)\},
\end{equation}
equipped with the product order, that is, 
$$(\sigma', \rho') \leq (\sigma, \rho) \Longleftrightarrow 
\sigma' \leq \sigma \text{ in }  \X, \rho \leq \rho' \text{ in } \Sigma.$$

We note at this point that throughout the following, given a polyhedral complex
$\X$ we regard it as a poset whose partial order is given by inclusion. 
We will sometimes refer to this as the \emph{face poset} of $\X$ to emphasize that
we forget all geometric information and only consider the poset structure.

\subsection{Real phase structures on polyhedral subspaces}

Given a polyhedron $\sigma$ in $\T \Sigma$ of sedentarity $\rho$, 
we define its \emph{tangent space} by
$T(\sigma) := T(\sigma \cap \T\OOO_\rho) \subset \T \OOO_\rho = \R(\rho).$
Its integer and binary versions are
\begin{align*} 
  T_\Z(\sigma) &:= T(\sigma) \cap \Z(\rho) \subset \Z(\rho), \\ 
	T_{\Z_2}(\sigma) &:= T_\Z(\sigma) \otimes \Z_2 \subset 
	\Z_2(\rho).
\end{align*}
For a vector space $V$ over $\Z_2 := \Z / 2\Z$,
we denote by $\Aff_d(V)$ the set of all affine subspaces in $V$ of dimension $d$.

\begin{definition}\label{def:evencovering}
A collection of subsets $S_1, \dots, S_k \subset V$ is called an \emph{even covering} if every element in $V$ is contained in an even number of the sets.
Equivalently, 
\[
  S_1 \triangle \dots \triangle S_k = \emptyset,
\]
where $S \triangle T := (S \cup T) \setminus (S \cap T)$ is the symmetric difference. 
\end{definition}

\begin{definition} \label{defGeneralRealPhaseStr}
Let $\mathcal{X}$ be a polyhedral complex of pure dimension $d$ and sedentarity $\rho$ contained in a tropical toric variety $\T \Sigma$. 
A \emph{real phase structure} on $\mathcal{X}$ is a map 
\[
  \mathcal{E} \colon \Facets(\X)  \to \Aff_d(\Z_2(\rho))
\]
satisfying
\begin{enumerate}
\item for every facet $\sigma$ of $\mathcal{X}$, the set $ \mathcal{E}(\sigma) $ is an affine subspace of $\Z_2(\rho)$ parallel to $\sigma$, in formulas, $T(\EEE(\sigma)) = T_{\Z_2}(\sigma)$; 
\item for every codimension one face $\tau$ of $\mathcal{X}$ with facets $\sigma_1, \dots , \sigma_k$ adjacent to it, the sets $\E(\sigma_1), \dots, \E(\sigma_k)$ are an even covering, 
\end{enumerate}

A real phase structure on a polyhedral subspace $X \subset \T \Sigma$ is
a real phase structure on some subdivision $\mathcal{X}$ of $X$. 
\end{definition}

\begin{remark} \label{RemSubdivision}
  If $\X'$ is a subdivision of $\mathcal{X}$, 
	there is a bijection between the set of real phase structures  on $\mathcal{X}$ and those on  $\mathcal{X}'$. 
	This bijection  is  given  by	$\E'(\sigma') = \E(\sigma)$ if $\sigma$ is the face of $\X$ containing $\sigma'$ in its relative interior.
	It is obvious that this assignment is injective, and 
	condition 2)  in Definition \ref{defGeneralRealPhaseStr} for a real phase structure on $\X'$ ensures that it is surjective. In fact, if a codimension one face $\tau$ is adjacent to exactly two facets $\sigma_1$ and $\sigma_2$, then it follows from condition 2) that $\E(\sigma_1)=\E(\sigma_2)$.
	Therefore, a real phase structure on a polyhedral subspace $X$ is a real phase structure on 
	any polyhedral complex $\X$ such that $|\X| = X$.
\end{remark}

\subsection{Tropical patchworking}\label{sec:tropicalPatchworking}

Following \cite[Chapter 11]{GKZ}, we now describe how to obtain a space homeomorphic to 
the real part of a complex toric variety by glueing together multiple symmetric copies of the tropical toric variety.
Let $\Sigma$ be a pointed  polyhedral fan in $\R^{n}$ defining a tropical toric variety $\T \Sigma$. 
For every $\varepsilon \in \Z_2^{n}$, we denote by  $\T \Sigma(\varepsilon)$ 
a copy of $\T \Sigma$ indexed by $\varepsilon$. We then define
\begin{equation}\label{eq:tropicalrealtoric}
\PW \T \Sigma:=  \bigsqcup_{\varepsilon \in \Z_2^{n}} \T \Sigma(\varepsilon) / \sim,  
\end{equation}
where $\sim$ identifies strata $\T\Star_\rho \Sigma(\varepsilon)$ and 
$\T\Star_\rho \Sigma(\varepsilon')$ if and only if $\varepsilon +  \varepsilon'$ is in 
$T_{\Z_2}(\rho)$.

We can also give a more precise description of  $\PW \T \Sigma$ as a direct limit of topological spaces. 
Consider the poset
\[
  Q(\Sigma) := \{(\rho, \varepsilon) : \rho \in \Sigma, \varepsilon \in \Z_2(\rho)\}
\]
with partial order given by 
\[
  (\rho, \varepsilon) \leq (\eta, \delta) :\Longleftrightarrow 
	\eta \subset \rho, \varepsilon = \pi(\delta),
\]
where $\pi \colon \Z_2(\eta) \to \Z_2(\rho)$ denotes
the canonical projection.
We construct a direct system of topological spaces over $Q(\Sigma)$ by
assigning to $(\rho, \varepsilon)$ a copy of the toric variety $\T\Star_\rho \Sigma$,
denoted by $\T\Star_\rho \Sigma(\varepsilon)$.
For $(\rho, \varepsilon) \leq (\eta, \delta)$, the corresponding map is given by the inclusion
$\T\Star_\rho \Sigma \subset \T\Star_\eta \Sigma$.
Then $\PW \T \Sigma$ is equal to the limit of this direct system. 
In particular, we note 
that $\PW \T \Sigma$ contains $2^{n-\dim \rho}$ copies of $\T\Star_\rho \Sigma$
(and hence of $\T\OOO_\rho$) labelled by $\varepsilon \in \Z_2(\rho)$.

\begin{remark} \label{rem:ExtendedLog}
  The construction can also be understood using the formalism of
	toric varieties over arbitrary semigroups, cf.\ \Cref{rem:Semigroups}.
	Consider the semigroup $S = \T \times \Z_2$. 
	The toric variety $S\Sigma$ is the disjoint union 
	$\bigsqcup_{\varepsilon \in \Z_2^{n}} \T \Sigma(\varepsilon)$
	appearing in \Cref{eq:tropicalrealtoric}.
	The quotient of $S$ under $\{-\infty\} \times \Z_2$ is called 
	the semigroup (more often, semifield or hyperfield) of \emph{real}
	tropical numbers which we denote by $\PW\T$, see \cite{ViroBasic}. As a set, 
	$\PW\T = \T \cup_{-\infty} \T$ is a union of two copies of $\T$ along $-\infty$. 
	Then $\PW\T\Sigma$ is
	in fact the toric variety associated to the semigroup $\PW\T$, 
	and the glueing described in \Cref{eq:tropicalrealtoric} is coming from
	the map $S\Sigma \to \PW\T\Sigma$ induced by the quotient homomorphism
	$S \to \PW \T$. 
	
	To explain the connection to the real toric variety $\R\Sigma$, first note that 
	$\log \colon (\Rnn, \cdot) \to (\T, +)$ 
	(with $\log(0) = -\infty$) is an isomorphism of semigroups. 
	Moreover, $\R$ can be described as the semigroup quotient of $\Rnn \times \Z_2$
	modulo $\{0\} \times \Z_2$. Hence we get an extended semigroup homomorphism
	$\log^\PW \colon \R \to \PW\T$. 
	These semigroup homomorphisms yield homeomorphisms
	$\Log \colon \Rnn\Sigma \to \T\Sigma$ and $\Log^{\PW} \colon \R \Sigma \to \PW\T\Sigma$. 
	We also note that this construction is a tropical reformulation of 
	for example \cite[Theorem 11.5.4]{GKZ}.
\end{remark}

\begin{example} \label{ex:realTn}
  We consider the tropical affine space $\T \Sigma = \T^n$.
	In continuation of the notation introduced \Cref{exAffineSpace},
	given $A \subset E = \{1, \dots, n\}$ and $\varepsilon \in \Z_2^{E \setminus A}$, 
	we denote by $\T^n_A(\varepsilon)$ the corresponding copy of $\T^n_A$ in 
	$\PW \T^n$.
	Instead of tuples $(A, \varepsilon)$, 
	we can alternatively label the strata of 
	$\PW \T^n$ 
	by the set of \emph{signed vectors} $C \in \{0,+1,-1\}^n$. 
	The connection is given by setting $A = E \setminus \text{supp}(C)$
	and $C|_{\text{supp}(C)} = (-1)^\varepsilon$. 
	The stratification agrees with the ordinary sign stratification
	of $\R^n$ under the homeomorphism $\Log^\PW \colon \R^n \to \PW \T^n$.
	For $i = 1, \dots, n$, we use the shorthand 
	$H_i := \bigcup_{\varepsilon \in \Z_2^{E \setminus \{i\}}} \T^n_{\{i\}}(\varepsilon) 
	\subset \PW \T^n$. Under $\Log^\PW$, this subset corresponds
	to the coordinate hyperplane $\{x_i = 0\} \subset \R^n$. 
\end{example}

Given a polyhedron $\sigma$ in $\T\Star_\rho \Sigma \subset \T \Sigma$ and 
$\varepsilon \in \Z_2(\rho)$, we let $\sigma(\varepsilon)$ 
denote the copy of $\sigma$ in $\T\Star_\rho \Sigma(\varepsilon) \subset \PW \T \Sigma$. 
We extend the notion of polyhedral complex to $\PW \T \Sigma$ 
be declaring the sets $\sigma(\varepsilon)$ to be the polyhedra of $\PW \T \Sigma$.

Let $\X$ be a  polyhedral complex 
of sedentarity $\rho$ in $\T \Sigma$ 
together with a real phase structure $\mathcal{E}$.
Given an arbitrary face $\tau \in \X$ of sedentarity $\eta$, we set
\[
  \EEE(\tau) = \bigcup_{\substack{\sigma \in \Facets(\X) \\ \tau \subset \sigma}}
	  \pi(\EEE(\sigma)),
\]
where $\pi$ denotes the canonical projection 
$\pi \colon \Z_2(\rho) \to \Z_2(\eta)$. 
Note that in general $\EEE(\tau)$ is not an affine subspace of $\Z_2(\eta)$.

\begin{defi}\label{def:realpart}
Let $\X$ be a  polyhedral complex in $\T \Sigma$ together with a real phase structure $\mathcal{E}$. 
The \emph{patchwork} $(\X,\EEE)$
is the polyhedral complex $\PW(\X,\mathcal{E})$ in $\PW \T \Sigma$ 
given by
\[
  \PW(\X,\mathcal{E}) = 
	  \{ \sigma(\varepsilon) : \sigma \in \X, \varepsilon \in \EEE(\sigma) \}.
\]
Let $X$ be a polyhedral subspace in $\T \Sigma$ together with a real phase structure $\EEE$. 
The \emph{patchwork} of $X$ with respect to $\EEE$ is the polyhedral subspace 
\[
  \PW (X,\mathcal{E}) := |\PW(\X,\mathcal{E})| \; \subset \PW \T \Sigma
\]
for some subdivision $\X$ of $X$. 
\end{defi}
If the real phase structure $\EEE$ is clear from the context, 
we allow ourselves to write $\PW\X$ and $\PW X$ instead of 
$\PW (\X,\mathcal{E})$ and $\PW (X,\mathcal{E})$.

\begin{remark} \label{remSubdivisionRealPart}
  Let $\X'$ be a subdivision of $\X$ with induced real phase structure $\EEE'$ 
	as explained in Remark \ref{RemSubdivision}. 
	Then clearly $\PW(\X',\EEE')$ is a subdivision of $\PW(\X,\EEE)$. 
	In particular, 
	$|\PW(\X,\EEE)| = |\PW(\X',\EEE')|$
	and hence $\PW (X,\mathcal{E})$ is well-defined. 
\end{remark}

\begin{proposition} \label{porRealPartClosed}
  Let $\X$ be a  polyhedral complex in $\T \Sigma$ together with a real phase structure $\mathcal{E}$. 
  The polyhedral complex $\PW(\X,\EEE) \subset \PW \T \Sigma$ represents a 
	closed cellular chain in the homology of  $\PW \T \Sigma$ over $\Z_2$. 
\end{proposition}

Note that the associated homology	class $[\PW(\X,\EEE)] \in H^{BM}_d(\PW \T \Sigma, \Z_2)$ is invariant under subdivisions of $\X$ by Remark \ref{remSubdivisionRealPart}. 
In particular, we can associate a well-defined homology class 
$[\PW(X,\EEE)] \in H^{BM}_d(\PW \T \Sigma, \Z_2)$ to any polyhedral subspace 
$X$ with real phase structure $\EEE$.

\begin{proof}
To see that $\partial \PW(\X,\EEE) = 0$ we need to verify that every  codimension one face of $\PW(\X,\EEE)$ is contained in an even number of 
facets of $\PW(\X,\EEE)$. 
If $\X$ is of sedentarity $\rho$, we can replace $\T \Sigma$ by 
$\T\Star_\rho \Sigma$,
hence it suffices to prove the statement for sedentarity $0$.

Let $\tau(\varepsilon)$ be a face of $\PW(\X,\EEE)$ of codimension one.
We first consider the case $\sed(\tau) = 0$. 
Then the facets of $\PW(\X,\EEE)$ containing $\tau(\varepsilon)$ are in 
bijection with the 
facets $\sigma$ of $\X$ containing $\tau$ and such that $\varepsilon \in \EEE(\sigma)$.
By Condition 2) of Definition \ref{defGeneralRealPhaseStr}, the number of such facets is even, which proves the claim in this case.

	Now suppose $\tau$ is of sedentarity $\rho$.
	Let $\sigma$ be a facet of $\X$ containing $\tau$. 
	This implies that $\RecCone(\sigma)$ intersects the
	relative interior of $\rho$ and moreover the intersection
	$\rho' := \RecCone(\sigma) \cap \rho$ 
	is a one-dimensional cone, see
	\cite[Proposition 3.2.6]{MikRau},  {\cite[Lemma 3.9]{ossermanRabinoff}}.
	Let $v \in \Z_2^n$ be the generator of $T_{\Z_2}(\rho') \cong \Z_2$.
	Since $v \in T_{\Z_2}(\rho)$, we have $\pi(v) = 0 \in \Z_2(\rho)$. 
	Since $v \in T_{\Z_2}(\sigma)$, by Condition 1) of 
	\Cref{defGeneralRealPhaseStr} the vector space parallel to $\EEE(\sigma)$ contains $v$. 
	Hence the assignment $\sigma(\varepsilon') \mapsto
	\sigma(\varepsilon' +v)$ defines a fixed-point-free involution
	on the set of facets of $\PW(\X,\EEE)$ containing $\tau(\varepsilon)$. 
	This proves that there is an even number of such facets, which proves the claim.
\end{proof}

\begin{remark}
  A tropical variety is commonly defined as a
	weighted $n$-dimensional rational polyhedral complex $\X$ 
	that satisfies the so-called balancing condition. 
	The  balancing condition 
	can be reformulated in terms of tropical homology. 
	From a weighted polyhedral complex $\X$ 
	one can construct its fundamental chain in tropical homology. 
	The balancing condition on $\X$  is then equivalent to 
	this fundamental chain being closed, see \cite[Proposition 4.3]{MZ}.
		
	In analogy to this picture, we propose to 
	think of Conditions 1) and 2) in Definition \ref{defGeneralRealPhaseStr}
	as a real balancing condition in tropical geometry.
	Indeed, given an arbitrary map 
	$\mathcal{E} \colon \Facets(\mathcal{X})  \to \Pow(\Z_2(\rho))$,
	we may still construct $\PW(\X,\EEE)$
	as a (cellular) Borel-Moore chain in $\PW \T \Sigma$. 
	Then Proposition \ref{porRealPartClosed}
	states that this chain is closed if $\EEE$ is a real phase structure.
	Moreover, the inverse is true 
	at least for certain natural special cases, for example
	when $\X \cap \R^n = \Sigma$.
	More precisely, the sedentarity $0$ part of the proof
	of Proposition \ref{porRealPartClosed} can
	be inverted to show that Condition 2) is equivalent to 
	$\PW(\X,\EEE)$ being closed in sedentarity $0$. 
	Moreover, assuming $\X \cap \R^n = \Sigma$, given $\rho \in \Sigma$ 
	one can extend the argument in the proof to show that 
	$\PW(\X,\EEE)$ is closed at higher sedentarity faces of $\overline{\rho}$
	if and only if $\EEE(\sigma)$ is a union of affine subspace satisfying Condition 1). 
	Allowing unions may be interesting when treating tropical varieties with weight functions which are not equal to one. Here we do not treat this generality for the sake of simplicity. 
\end{remark}

\section{{Non-singular tropical subvarieties and their patchworks }} 
\label{sec:nonsingular}

\subsection{Tropically non-singular polyhedral subspaces}
The local models for tropically non-singular polyhedral spaces are matroid fans. 
We recall them for convenience. 
Given a loopfree matroid $M$ on the base set $E$, 
we denote by $\Sigma_M$ the \emph{affine} matroid fan in
$\R^E$ and by $\projFan_M = \Sigma_M/\langle (1,\dots,1) \rangle$ 
the \emph{projective} matroid fan in $\R^E / \langle (1,\dots,1) \rangle $. 
We now describe how to construct both of these fans 
following Ardila and Klivans \cite{ArdilaKlivans}.
Fix the vectors $v_i = -e_i$ where $\{e_1, \dots, e_n\}$ 
is the standard basis of $\R^E$ for 
$E = \{1, \dots, n\}$ and set $v_I = \sum_{i \in I} v_i$ 
for any subset $I \subset E$.
For a chain of flats 
$$\F =\{ \emptyset \subsetneq F_1 \subsetneq F_2 \subsetneq \dots \subsetneq F_k \subsetneq E\}$$
in the lattice of flats $\L$,
define the $k+1$-dimensional cone 
\[\sigma_{\F} = \langle v_{F_1}, \dots , v_{F_k}, \pm v_E \rangle_{\geq 0}.\]
The affine  matroid fan $\Sigma_M$ is the collection of all such cones ranging over the chains in $\L$. 
In particular, the top dimensional faces of  $\Sigma_M$ are in one to one correspondence with 
the maximal chains in the lattice of flats of $M$. 
The projective matroid fan $\projFan_M$ is the image of 
$\Sigma_M$ in the quotient $\R^E / (1, \dots, 1)$. 

If a matroid $M$ has loops $L = \cl(\emptyset)$, then we set $\Sigma_M := \Sigma_{M/L} \subset \R^{E \setminus L}$ and 
$\projFan_M := \projFan_{M/L} \subset \R^{E \setminus L}/(1,\dots,1)$. 
We will often assume $E=\{0,\dots, n\}$. 
Note that by \Cref{exAffineSpace} we can consider $\R^{E \setminus L} = \T\OOO_{\rho_L}$
as a torus orbit of $\T^{n+1}$, and $\R^{E \setminus L}/(1,\dots,1) = \T\OOO_{\rho_L}$
as a torus orbit of $\TP^n$ (excluding the trivial case $E = L$). 
In this sense, in the following we will regard $\Sigma_M$ and $\projFan_M$ as subsets
of $\T^{n+1}$ and $\TP^n$, respectively.

\begin{definition}\cite[Chapter 6]{MikRau} \label{def:nonsingular}
  Let $\Sigma$ be a fan and $\T\Sigma$ the associated tropical toric
	variety. Let
	$X \subset \T\Sigma$ be a polyhedral subspace and consider $p \in X$ 
	of sedentarity $\rho \in \Sigma$ and $\sed(p) = \dim(\rho) = k$. 
	 
	Then $X$ is \emph{non-singular} at $p$ if there exist	
	\begin{itemize}
		\item 
			a toric isomorphism $\psi \colon \T U_\rho \to \T^k \times \R^{n-k} = 
				\T U_{\rho_{\{1,\dots, k\}}} \subset \TP^n$, 
		\item 
			an open neighbourhood $p \in U \subset \T U_\rho$,
			\item 
			a matroid $M$ on $E = \{0,1,\dots, n\}$,
	\end{itemize} 
	such that 
 $\psi(p) = (-\infty, 0)$  and  
	$\psi(X \cap U) = \overline{\projFan}_M \cap \psi(U)$.

	The polyhedral subspace $X$ is a \emph{non-singular tropical subvariety} of $Y$ if it is non-singular at all its points. 
\end{definition}

\subsection{Patchworks of non-singular subvarieties are topological manifolds }

Our next goal is to prove \Cref{thm:PLman} from the introduction. We point the reader to \cite{RRSmat} for the background on the connection between  oriented matroids  and real phase structures. 

\medbreak 
\noindent{\bf Theorem \ref{thm:PLman}.} 
{\it Let $X \subset \T \Sigma$ be a non-singular tropical subvariety 
	of a tropical toric variety $\T \Sigma$. 
	Let $\EEE$ be a real phase structure on $X$. 
	Then the patchwork $\PW X$ is a topological manifold.} 
\medbreak

By  \cite[Theorem 1.1]{RRSmat}, we proved that matroidal fans equipped with real phase structures are equivalent to oriented matroids. 
In \Cref{cor:matroidPLmanifold}, we show that the  patchwork  of a matroid fan with a real phase structure is a topological manifold with determined topology. This is essentially a corollary of the Folkman-Lawrence topological representation theorem for oriented matroids \cite{FolkmanLawrence}. 
We note that similar results on patchworking of matroids were obtained in \cite[Theorem 3.4]{AKW} and in \cite[Theorem A]{celaya2022patchworking}.
We start with two examples  describing the faces of the 
simplicial complexes $\PW\overline{\affFan}_M$ and $\PW\overline{\projFan}_M$.

\begin{example} \label{ex:PosetCovectors}
  Let $\MMM$  be an oriented matroid with underlying matroid $M$ and 
	let $\EEE$ be the associated real phase structure on $\affFan_M$ \cite[Definition 7]{RRSmat}. Hence $\EEE$  is also a real phase structure on  $\overline{\affFan}_M \subset \T^n$.  
	By combining Example  \ref{ex:realTn} with the definition of $\EEE$, 
	we see that the faces of the polyhedral complex 	$\PW \overline{\affFan}_M \subset \PW\T^n$ are labelled by the chains 
	of covectors $0 < Z_1 < \dots < Z_l$ from  $\MMM$.
	The final covector $Z_l$ determines the stratum of $\PW \T^n$ that contains the face. 
	The associated chain of flats $F_i = \phi(Z_{l-i})$ determines the corresponding 
	face of $\overline{\Sigma}_M$. 
	The map $\phi$ here denotes the forgetful map from oriented matroids to 
	matroids sending a covector $X$ to the flat $\mathrm{Supp}(X)^c$, 
	see for example \cite[Section 3.1]{RRSmat}. 
	In particular, the intersection $H_i \cap \PW \overline{\affFan}_M$ 
	is the subcomplex corresponding to all chains such that
	$i \in \phi(Z_l)$. 
\end{example}

\begin{example} \label{ex:LasVergnas}
  Let $\MMM$  be an oriented matroid with underlying matroid $M$ and let 
	$\EEE$ be the associated real phase structure on $\projFan_M\subset\R^{n+1} / (1, \dots, 1)\simeq \R^n$.
  Again, we denote by the same letter the induced real phase structure on 
	$\overline{\projFan}_M \subset \TP^n$.

	Pick $\varepsilon \in \Z_2^{n+1}$ such that $T = (-1)^{\varepsilon}$ is a tope of $\MMM$.
	We denote by $\TP^n(\varepsilon) \subset \PW\T\PP^n$ the copy of $\TP^n$ in $\PW\T\PP^n$
	obtained as the projection of $\T^{n+1}(\varepsilon)$. 
	We are interested in the polyhedral complex 
	$\PW \overline{\projFan} \cap \partial \TP^n(\varepsilon)$. 
	By 	Examples \ref{exAffineSpace} and \ref{ex:PosetCovectors},
	it is a simplicial complex whose faces are labelled by 
	chains of covectors $0 < Z_1 < \dots < Z_l < T$ or, in other words,
	chains of flats $\emptyset \subsetneq F_1 \subsetneq \dots \subsetneq F_l \subsetneq E$
	which are 
	touching $T$. 
	The lattice of flats touching $\varepsilon$ is known as the 
	\emph{Las Vergnas lattice} of $\MMM$ with respect to $T$. 
	We denote it by $\FFF_{lv}^{\MMM}(T)$. 
	Let us also remark that by \cite[Theorem 3.4]{AKW}, 
	the simplicial complex $\PW\overline{\projFan}_M \cap \partial \T\PP^n(0)$
	is equal to the \emph{positive Bergman complex} from \emph{loc.~cit.},
	though rather than intersection with a standard sphere, 
	we intersect with  $\partial \T\PP^n(0)$ which is homeomorphic to a sphere. 
\end{example}

The simplicial complexes mentioned in the previous examples have been studied by Folkman and Lawrence in the context of the 
topological representation theorem for oriented matroids. We translate the results to our setup in the following proposition.
To state the results in the conventional way, it is useful to intersect with a sphere $S^{n-1} \subset \PW\T^n$. 
A natural choice for $S^{n-1}$ is given by taking, in each stratum $\T^n(\varepsilon)$ of $\PW\T^n$, the symmetric copy of the  boundary of the negative orthant, that is $$S(\varepsilon) = \{ x \in \T^n(\varepsilon) \ | \  \forall \ i \ x_i \leq 0 \text{ and } \exists \ i : \ x_i = 0 \}.$$ The union of all of the $S(\varepsilon) $ in $\PW\T^n$ is a polyhedral sphere of dimension $n-1$, see  \Cref{fig:polyhedral sphere}. In fact, it is the boundary of a cube.
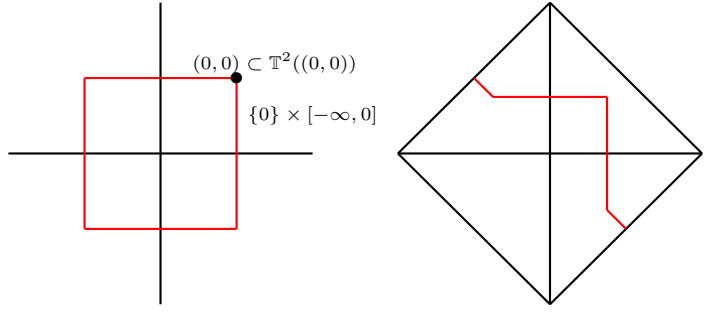
\begin{figure}
\begin{tikzpicture}
    \draw[thick, -] (-2,0) -- (2,0) ;
    \draw[thick, -] (0,-2) -- (0,2) ;
    \draw[thick, red] (-1,-1) -- (1,-1) ;
     \draw[thick, red] (1,-1) -- (1,1) ;
    \draw[thick, red] (1,1) -- (-1,1) ;
    \draw (1,1) node{$\bullet$};
    \draw (1.5,1.2) node{\tiny{$(0,0)\subset \T^2((0,0))$}};
    \draw (2,0.5) node{\tiny{$\lbrace 0\rbrace \times \left[ -\infty,0 \right]$}};
    \draw[thick, red] (-1,1) -- (-1,-1);
\end{tikzpicture}
\begin{tikzpicture}
    \draw[thick, -] (-2,0) -- (2,0) ;
    \draw[thick, -] (0,-2) -- (0,2) ;
    \draw[thick, -] (2,0) -- (0,2) ;
     \draw[thick, -] (2,0) -- (0,-2) ;
        \draw[thick, -] (-2,0) -- (0,-2) ;
           \draw[thick, -] (-2,0) -- (0,2) ;
    \draw[thick, red] (0.75,0.75) -- (0,0.75) ;
      \draw[thick, red] (-0.75,0.75) -- (0,0.75) ;
           \draw[thick, red] (-0.75,0.75) -- (-1,1) ;
       \draw[thick, red] (0.75,-0.75) -- (0.75,0) ;
        \draw[thick, red] (0.75,-0.75) -- (1,-1) ;
     \draw[thick, red] (0.75,0.75) -- (0.75,0) ;
\end{tikzpicture}
\caption{On the left is a polyhedral sphere (drawn in red) in $\PW\T^2$. On the right is a patchwork of $\PW\overline{\projFan}_M$ (drawn in red)  in $\R\PP^{2}$ for the uniform matroid $U_{2,3}$. The real projective plane is obtained from the square by identifying antipodal points on the boundary.} \label{fig:polyhedral sphere}
\end{figure}

\begin{proposition} \label{cor:matroidPLmanifold}
    Let $M$ be a matroid of rank $d$ on $n$ elements and let $\EEE$ be a real phase structure on $\affFan_M$.
		\begin{enumerate}
			\item The complexes $\PW \overline{\affFan}_M \cap S^{n-1}$, $\PW\overline{\projFan}_M$ and $\PW \overline{\affFan}_M$ are topological  manifolds 
			      and homeomorphic to $S^{d-1}$, $\R\PP^{d-1}$ and $\R^d$, respectively. 
			\item For a sign vector   
			$\varepsilon \in \cup_{\sigma \in \Sigma_M} \EEE(\sigma)$,
			  the complexes $\PW\overline{\projFan}_M \cap \partial \T\PP^n(\varepsilon)$ and $\PW\overline{\projFan}_M \cap \T\PP^n(\varepsilon)$
			      are topological  manifolds (with boundary in the second case) and 
						homeomorphic to $S^{d-2}$ and $D^{d-1}$, respectively.
			\item The subcomplexes $H_i \cap \PW \overline{\affFan}_M \cap S^{n-1} \subset \PW \overline{\affFan}_M \cap S^{n-1}$, $i = 1, \dots, n$, 
			      form a pseudosphere arrangement. The associated oriented matroid is equal to $\MMM_\EEE$. 
		\end{enumerate}
\end{proposition}

\begin{proof}
  Let $\MMM = \MMM_\EEE$ denote the oriented matroid constructed in  \cite[Section 3.3]{RRSmat}.
	By \cite[Theorem 1.1]{RRSmat} and Example \ref{ex:PosetCovectors}, 
	the simplicial complex $\PW \overline{\affFan}_M \cap S^{n-1} \subset \PW\T^n$ 
	is a geometric realization 
	of the order complex of the poset $\CCC$ of covectors of $\MMM$. 
	Analogously, by Theorem \cite[Theorem 1.1]{RRSmat} and Example 
	\ref{ex:LasVergnas}, given $\varepsilon$ such that $T = (-1)^\varepsilon$ is tope of $\MMM$, 
	the simplicial complex $\PW\overline{\projFan}_M \cap \partial \T\PP^n(\varepsilon)$
	is the order complex of the poset $\FFF_{lv}^{\MMM}(T) \setminus \{0,T\}$.  
	The statements then follow from the Folkman-Lawrence topological representation theorem \cite{FolkmanLawrence}. 
	See also
	\cite[Theorems 4.3.3, 4.3.5, 5.2.1]{bjorner}.
\end{proof}

\begin{proof}[Proof of \Cref{thm:PLman}] 
  Fix a point $p \in X$. 
  Let $\psi \colon \T U_\rho \to \T^k \times \R^{n-k}$, $p \in U \subset \T U_\rho$ and $M$
	be the objects guaranteed by the definition of $X$ being non-singular at $p$. 
	Since $(-\infty, 0) \in \overline{\PP\Sigma}_M$, 
	the set $F = \{1, \dots, k\}$ is a flat of $M$.
	After replacing $M$ by $F \oplus M/F$ (and possibly shrinking $U$), 
	we may assume that $v_F$ is	contained in the lineality space of $\PP\Sigma_M$. 
	In particular, any facet of $\PP\Sigma_M$  (in the coarse sense)
	intersects $\psi(U)$. It follows that $\psi$ induces a real phase structure from $X \cap U$ on $\PP\Sigma_M$. 
	
	Given an open subset of a toric variety $U \subset \T \Sigma$, 
	we denote by $\PW U \subset \PW \T \Sigma$ the open subset obtained
	by copying $U$ in each stratum $\T \Sigma(\varepsilon) \subset \PW Y$. 
	We consider the open subset $\PW X \cap \PW U \subset \PW X$. By definition, under $\psi$ (or rather, its extension to $\PW U_\rho \to \PW(\T^k \times \R^{n-k})$)
	it is homeomorphic to $\PW \overline{\PP\Sigma}_M \cap \psi(\PW U)$. 
	Since $\PW \overline{\PP\Sigma}_M$ is a 
	topological manifold by Corollary \ref{cor:matroidPLmanifold}, 
	it follows that $\PW X$ is a topological 
	manifold at all points in $\PW X \cap \PW U$. Since $p$ was arbitrary, these sets cover $\PW X$, 
	hence the claim follows. 
\end{proof}

\section{Patchworking close to the tropical limit}

\label{sec:patchworkinglimit}

\subsection{Non-singular tropical limits} \label{subsection NSTL}

We denote by $\KK = \C((t))$ the field of 
formal Laurent series with complex coefficients,
and by $\KK_\R$ the subfield of series with real coefficients. 
As algebraic closure we fix the field of Puiseux series $\bar{\KK} = \C\{t\}$.
Throughout the following, we fix an algebraic variety $\XX \subset (\KK^*)^n$ defined over 
the subring in $\KK$ of series that converge for $0 < |t| < 1$.
We can alternatively regard $\XX$ as a 
subset of $(\C^*)^n \times \D^*$ describing a family 
of algebraic varieties in $(\C^*)^n$ parametrized by $\D^*$, 
the punctured unit disc in $\C$. 
From now on, we refer to such $\XX$ as an 
\emph{analytic family over $\D^*$}. 
For $t \in \D^*$, we denote the fibre by $\XX_t = \XX \cap ((\C^*)^n \times \{t\})$.
We define the tropical limit $\Trop(\XX) \subset \R^n$ as the closure of $-\val(\XX \otimes \bar{\KK})$, where 
\[
  \val \colon (\bar{\KK}^*)^n \to \R^n
\]
is the coordinate-wise valuation map. 
The tropicalization  $\Trop(\XX)$ is a tropical subvariety 
(that is, a weighted polyhedral subspace that 
satisfies the balancing condition). 
We set $X = \Trop(\XX)$.

For each $p \in X$, we have an associated initial scheme 
$\init_p \XX \subset (\C^*)^n$.
We have $d := \dim_\KK \XX = \dim_\C \init_p \XX$.
For a generic point $p \in X$, $\init_p \XX$ 
is invariant under the action of an $d$-dimensional
subtorus $\Tbold_p$ of $(\C^*)^n$. 
We denote the quotient under this action by 
$\XX_p = \init_p \XX / \Tbold_p$. Hence $\XX_p$ is a $0$-dimensional scheme. 
The \emph{weight} of a generic point $p$ is defined
as 
$\omega(p) = 
\deg(\XX_p) = \dim_\C  K[\XX_p]$.
In particular, the weight of $p$ is $1$ 
if and only if 
$\XX_p$ is a reduced
point if and only if
$\init_p \XX$ is reduced and irreducible. 
Throughout the following, and by slight abuse of notation,
we say that \emph{$\XX$ has non-singular tropical limit}, 
or just \emph{$X = \Trop(\XX)$ is non-singular}, if both of the following conditions hold:
\begin{itemize}
	\item the scheme $\init_p \XX$ is reduced and irreducible for generic $p$ (weight $1$), 
	\item the polyhedral subspace $X$  is non-singular 
	      in the sense of Definition \ref{def:nonsingular}.
\end{itemize}

We first collect a few consequences of $\XX$ having non-singular tropical limit. 
A subscheme $\bL \subset (\C^*)^n$ is called \emph{linear} 
if its closure $\overline{\bL} \subset \PP^n$ 
under  the  standard inclusion  of $(\C^*)^n$ to $\PP^n$ 
is a linear subspace. 
In particular, this implies that $\bL$ is reduced and irreducible.  
Then $\overline{\bL}  \backslash \bL$ is the complement of an 
arrangement of $n+1$ (not necessarily distinct) hyperplanes. 
We denote the associated matroid by $M_\bL$.

\begin{lemma} \cite[Proposition 4.2]{KatzPayne}\label{lemmaline}
  Let $\bL \subset (\C^*)^n$ be a subscheme 
	with non-singular tropical limit $\Trop(\bL) = \PP\Sigma_M$,
	where $M$ is a matroid. 
	Then $\bL$ is linear and $M_\bL = M$. 
\end{lemma}

A polyhedral subdivision $\X$ of $X=\Trop(\XX)$ is called \emph{tropical}
if it satisfies the following two conditions.
\begin{enumerate}
		\item 
		  Fix $\sigma \in \X$. Then for any $p,q$ in the relative
		  interior of $\sigma$ we have $\init_p \XX = \init_q \XX$. 
		  We set $\init_\sigma \XX = \init_p \XX$ for one such $p$. 
		\item 
			The scheme $\init_\sigma \XX$ is invariant under the action 
			of the subtorus $\Tbold_\sigma = T(\sigma) \otimes \C^* \subset (\C^*)^n$. 
			We denote the quotient by
			$\XX_\sigma = \init_\sigma \XX / \Tbold_\sigma$. 
			We have $\dim_\C \XX_\sigma = d - \dim \sigma$.
\end{enumerate}

It is shown for example in \cite[Theorem 2.2.1]{SpeyerThesis} that tropical subdivisions 
exist for any analytic family $\XX$. 
Moreover, we have the following. 

\begin{proposition} \label{prop:SchoenPropertiesI}
  If $\XX$ has  non-singular tropical limit $X = \Trop(\XX)$, then any subdivision $\X$ of $X $
	is tropical. 
\end{proposition}

\begin{proof}
  This is a result known more generally for so-called schön subvarieties, 
	see \cite[Definition 1.3, Theorem 1.4]{Tev-CompactificationsSubvarietiesTori}
	for the \enquote{constant coefficient} case and 
	\cite[Definition 7.1, Theorem 7.12]{LQ-SomeResultsTropical} for the general case.
	By \cite[Lemma 2.7]{Hacking08} and \cite[Theorem 2.4.2]{SpeyerThesis}, 
	the variety $\XX$ is schön if and only of 
	$\init_p \XX$ is non-singular for all $p \in X$. To finish the proof we show the
	latter property (see also 
	\cite[Section 2, Propositions 2.6 and 2.7]{KS-TropicalGeometryMotivic}
	for a summary of the previous discussion). 
	By \cite[Proposition 2.2.3]{SpeyerThesis}, the tropical limit of $\init_p \XX$ for $p \in X$ 
	is $\Trop(\init_p \XX) = \Star_p X$.
	By our assumptions, $\Star_p X$ is isomorphic to a matroid fan 
	$\PP\Sigma_M$ via a change of coordinates in $\GL(n, \Z)$.
	Hence by \Cref{lemmaline}, $\init_p \XX$ is isomorphic to a linear subvariety via 
	a change of coordinates in $\GL(n, \Z)$. In particular, $\init_p \XX$ is non-singular.
\end{proof}

\begin{remark}
In the situation of Theorem \ref{Patchworking} and Corollary \ref{cor:boundsSemiStable},  the tropical subvarieties appearing are those whose vertices in $\R^N$ have rational coordinates, and have been called  
$\QQ$-rational tropical varieties in \cite{IKMZ}. 
This is because  $\val(\bar{\KK}) = \QQ$. These have also been called rational tropical varieties (we use $\QQ$ to avoid confusion with rational varieties  in algebraic geometry or  with rational polyhedral complexes). We do not know of an example of a tropical subvariety in a tropical toric variety which cannot be deformed to a $\QQ$-tropical variety. 
 \end{remark}

\subsection{Real tropical limits and patchworking}

Let $\XX$ be an analytic family over $\D^*$. From now on,
we additionally assume that $\XX$ is \emph{real}, that is, 
defined over the subfield $\KK_\R = \R((t))$ of real Laurent series.
Equivalently, $\XX \subset (\C^*)^n \times \D^*$ is invariant
under the canonical conjugation on $(\C^*)^n \times \D^*$.
In particular, for any 
$t \in \D^* \cap \R$, $\XX_t \subset (\C^*)^n$ is a real variety 
(invariant under conjugation on $(\C^*)^n$).
Moreover, all initial varieties $\init_p \XX \subset (\C^*)^n$ are real.
We want to use the real structure on the analytic family to define a real phase structure
on $X = \Trop(\XX)$. 

\begin{definition}\label{def:realphaselinear}
  Let $\XX$ be a real analytic family over $\D^*$, 
	$X = \Trop(\XX)$ and $\X$ a tropical subdivision of $X$. 
	We define the map 
	$ \EEE_\XX \colon \Facets(\X)  \to 
	{\Pow}(\Z_2^n)$
	as follows: The set $\EEE_\XX(\sigma)$ 
	is the set of orthants of 
	$(\R^*)^n$ meeting $\R \init_\sigma \XX$,
	\[
		\EEE_\XX(\sigma) :=
		  \left\{ \varepsilon\in\Z_2^n : 
			  (-1)^\varepsilon \cdot \R \init_\sigma \XX \cap \Rpos^n 
			\neq \emptyset \right\}.
	\]
\end{definition}

For general $\XX$, this definition might not produce a real phase structure in the sense of Definition \ref{defGeneralRealPhaseStr} since
$\EEE(\sigma)$ could be empty or a union of several affine subspaces. 
However, our first goal is to show that if $\XX$ has non-singular tropical limit $X$
then $\EEE_\XX$ is a real phase structure on $X$.

Recall that the projective matroid fan $\mathbb{P} \Sigma_M$  of $M$ is a projection of the affine matroid fan $\Sigma_M$ of $M$ along the direction $(1, \dots, 1)$. 
Given an oriented matroid $\MMM$ over $M$, the real phase structure $\EEE_{\MMM}$ on    $\Sigma_M$ from \cite[Definition 7]{RRSmat} 
can also be projected along this direction to produce a real phase structure  on $\mathbb{P} \Sigma_M$. 
Given a real linear subvariety $\bL \subset (\C^*)^n$, we denote by $\MMM_\bL$ the oriented matroid associated to the (cone over the) real hyperplane arrangement
$\R\overline{\bL} \setminus \R \bL$. 
The next lemma compares the real phase structure on $\Trop(\bL) = \PP\Sigma_{M_\bL}$ from  \Cref{def:realphaselinear} with the one
coming from $\MMM_\bL$.

\begin{lemma} \label{lemmalinear}
  Let $\bL \subset (\C^*)^n$ be a real and linear subvariety of the torus.
	Then the real phase structure $\EEE_{\bL}$ from Definition \ref{def:realphaselinear} 
	on $\Trop(\bL) = \PP\Sigma_{M_L}$
	coincides with the projection to $\PP\Sigma_{M_L}$ of the real phase structure $\EEE_{\MMM_\bL}$ from \cite[Definition 7]{RRSmat}.  
\end{lemma}

\begin{proof}
For  $\bL \subset (\C^*)^n$  linear and real, 
we denote by $C\bL$ the cone over $\bL$, that is, the preimage 
of $\bL$ under $(\C^*)^{n+1} \to (\C^*)^n$, $(z_0, \dots, z_n) \mapsto (z_1/z_0, \dots, z_n /z_0)$. 
We denote by $\overline{C \bL}$ the closure of $C\bL$ in $\C^{n+1}$. 
Then $\R\overline{C\bL}  \backslash \R C\bL$ defines a 
central real hyperplane arrangement which 
yields an oriented matroid $\MMM_\bL$.
  We set $M = M_\bL$ and $\MMM = \MMM_\bL$. 
  A facet $\sigma$ of $\Sigma_{M}$ is labelled by a complete flag of flats $\F =  \{\emptyset = F_0 \subsetneq \dots \subsetneq F_k =E\}$ in $M$. 
	It is straightforward to check that the initial variety 
	$\init_\sigma C \bL$ is another (central) linear subvariety whose associated oriented matroid is
	$\MMM_\F := \bigoplus_{j = 1}^k \MMM|F_j/F_{j-1}$. 
	Hence, by definition, $\varepsilon \to (-1)^\varepsilon$ gives a bijection from $\EEE_{C\bL}(\sigma)$ to the 
	tope vectors of $\MMM_\F$.  
	However, it is explained in the proof of \cite[Lemma 3.1]{RRSmat}
	that the same statement is true for $\EEE_{\M}(\sigma)$
	as defined in \cite[Definition 7]{RRSmat}.
\end{proof}

\begin{proposition} \label{prop:limitrealstructure}
  Let $\XX$ be a real analytic family over $\D^*$ with non-singular tropical
	limit $X = \Trop(\XX)$. 
	Then the map $\EEE_\XX$ is a real phase structure on $X$.
\end{proposition}

\begin{proof}
	For a facet $\sigma \in \X$, the associated initial scheme $\init_\sigma \XX$ 
	is reduced and irreducible.
	Hence, by Condition (2) of a tropical subdivision, $\init_\sigma \XX$ is
	a translation of the subtorus $\Tbold_\sigma$. 
	It follows that $\EEE(\sigma)$ 
	is an affine space in $\Z_2^n$ tangent to $T_{\Z_2}(\sigma)$,
	which proves Condition (1) of a real phase structure. 
	To check Condition (2), 
	up to change of coordinates (as in the proof of \Cref{prop:SchoenPropertiesI})
	and using \Cref{lemmaline}, we can assume that $\XX$ is real and linear.
	Condition (2) then follows from \Cref{lemmalinear}
	and \cite[Theorem 1.1]{RRSmat}. 
\end{proof}

Our aim in the remainder of this section is to prove the following theorem from the introduction. 

\medbreak
\noindent {\bf Theorem \ref{Patchworking}} (Patchworking for non-singular tropical limits){\bf .} 
{\it Let $\XX$ be a real analytic family over $\D^*$
	with  non-singular tropical limit $X = \Trop(\XX)$ 
	and associated real phase structure $\EEE = \EEE_\XX$. 
	Let $\X$ be a subdivision of $X$ and 
	$\Sigma$ a pointed unimodular fan such that $\RecCone(\X) \cup \Sigma$ is a fan. 
	
	Then for sufficiently small and positive 
	$t \in \D^* \cap \R$ 
	the pairs $(\R\Sigma, \R \overline{\XX}_t)$ and 
	$(\PW\T\Sigma, \PW(\overline{X}, \EEE))$ are homeomorphic.
	Moreover, the homeomorphism can be chosen to respect
	the stratification of $\R\Sigma$ and $\PW\T\Sigma$ by torus orbits. } 
\medbreak

Here, the closures are taken with respect to the partial compactifications
$(\C^*)^n \subset \C\Sigma$ and $\R^n \subset \T\Sigma$.

\subsection{Reduction to the unimodular case}

The first step in the proof of \Cref{Patchworking} is to construct out of the given data a toric degeneration of the variety $\C \Sigma$ such that the closure of $\XX$ in this family is a semi-stable degeneration. 
To do so, we follow closely \cite{NishinouSiebert} and \cite{Helm-Katz}.

Given  a polyhedron $\sigma \in \R^n$, we denote by $C(\sigma)$
the polyhedral cone in $\R^{n+1}$ obtained as the closure of the cone over $\sigma \times \{1\}$.
Given a polyhedral complex $\X$ in $\R^n$ such that $\Sigma = \RecCone(\X)$ is a fan, 
we denote by $C(\X)$ the fan in $\R^n \times \Rnn$ given by the cones
$C(\sigma)$, $\sigma \in \X$, and $\rho \times \{0\}$, $\rho \in \Sigma$. 

 We say $\sigma$ is \emph{unimodular} if $C(\sigma)$ is unimodular. In particular, the polyhedron $\sigma$ must have $\QQ$-coordinates. 
We call $\sigma$ \emph{strongly unimodular} if there exist $x_0 \in \Z^n$ and $e_1, \dots, e_k, f_1, \dots, f_l \in \Z^n$ a part 
of a lattice basis such that  $\sigma = \kappa + \rho$,  
\begin{align} 
 \kappa = \text{Conv}(x_0, x_0 + e_1, \dots, x_0 +e_k),  \text{ and } & & \rho = \langle f_1, \dots, f_l\rangle_{\geq 0}. 
\end{align}
This is equivalent to the condition that $C(\sigma)$ is unimodular and that the primitive generators of the rays of $C(\sigma)$ 
have last coordinate equal to $x_{n+1} = 0$ or  $x_{n+1} = 1$. In particular, the vertices of $\sigma$ must have $\Z$-coordinates. 
A  polyhedral complex $\PPP$ in $\R^n$ is called strongly unimodular 
if all its polyhedra are strongly unimodular.
For convenience, we also require that $\PPP$ contains a $0$-dimensional cell.
If  $\PPP$ is a strongly unimodular subdivision of $\R^n$, then 
$\Sigma = \RecCone(\PPP)$ is a pointed unimodular complete fan 
of $\R^n$
\cite{BurgosGil}. Our additional requirement that $\PPP$ contains a $0$-dimensional cell is made to ensure that $\Sigma$ is pointed. 
Moreover, $C(\PPP)$ is a pointed unimodular fan with support $\R^n \times \Rnn$.
We will deduce the general Patchworking Theorem \ref{Patchworking}
from the following  version.

\begin{thm}[Unimodular Patchworking] \label{UnimodularPatchworking}
	Let $\XX$ be a real analytic family 
	with  non-singular tropical limit $X = \Trop(\XX)$ 
	and associated real phase structure $\EEE = \EEE(\XX)$. 
	Let $\PPP$ be a strongly unimodular subdivision of $\R^n$
	such that 
	\[
	  \X = \{\sigma \in \PPP : \sigma \subset X\}
	\]
	is a subdivision of $X$.
	Set $\Sigma = \RecCone(\PPP)$.

	Then for sufficiently small and positive $t \in \D^* \cap \R$ the pairs 
	$(\R\Sigma, \R \overline{\XX}_t)$ and 
	$(\PW\T\Sigma, \PW(\overline{X}, \EEE))$ are homeomorphic.
	Moreover, the homeomorphism can be chosen to respect
	the stratification of $\R\Sigma$ and $\PW\T\Sigma$ by torus orbits. 
\end{thm}

The reduction from \Cref{Patchworking} to \Cref{UnimodularPatchworking}
is based on the following lemma. 

\begin{lemma} \label{lem:Refinements}
  Let $\Sigma$ be a pointed fan in $\R^n$ and $\X$ a polyhedral complex in $\R^n$ such that $\RecCone(\X) \cup \Sigma$ is a fan.
	\begin{enumerate}
		\item 
		  There exists a complete subdivision $\PPP$ of $\R^n$ 
		  such that $\X \subset \PPP$ and $\Sigma \subset \RecCone(\PPP)$. 
		\item 
		  Assume that $\Sigma$ is unimodular and pick a $\PPP$ 
			as in the previous item. 
			Then there exists an integer $d$ and a refinement 
			$\PPP'$ of $d \PPP$ which is strongly unimodular 
			and such that $\Sigma \subset \RecCone(\PPP')$.
	\end{enumerate}
\end{lemma}

\begin{proof}
  (1) Since $\RecCone(\X) \subset \Sigma$, the cones $C(\sigma)$, 
	$\sigma \in \X$, together with the cones $\rho \times \{0\}$, $\rho \in \Sigma$
	    form a pointed fan $\widetilde{\Sigma}$ in $\R^{n+1}$. 
			By 
			\cite[Theorem III.2.8]{Ewa-CombinatorialConvexityAlgebraic},
			any fan can be completed to a complete fan. 
			Let $\widetilde{\Sigma}'$ be a completion of $\widetilde{\Sigma}$. 
			Then it is straightforward to check that 
			$\widetilde{\Sigma}' \cap (\R^n \times \{1\})$ is a 
			complete subdivision with the desired properties. 
			
	(2) This is very similar to \cite{Helm-Katz}. 
	We include a short proof streamlined for our purposes. 
	    Let $\PPP$ be a subdivision from the previous item 
			and denote by $C(\PPP)$ the correponding fan in $\R^{n+1}$ 
			with support $\R^n \times \R_{\geq 0}$. 
			By \cite[Section 2.6]{Fulton} there exists a unimodular fan $\Sigma'$ 
			that refines $C(\PPP)$ without subdividing any of the 
			unimodular cones of $C(\PPP)$, in particular, the cones
			$\rho \times \{0\}$, $\rho \in \Sigma$. 
			Moreover, let $d$ be a common multiple of the (non-zero)
			last coordinates of primitive generators of all the rays in $\Sigma'$. 
			Then $\PPP' =\Sigma' \cap (\R^n \times \{d\})$ 
			is a strongly unimodular subdivision that refines $d \PPP$ and 
			such that $\Sigma \subset \RecCone(\PPP')$.
\end{proof}

\begin{proof}[Proof of \Cref{Patchworking} based on 
\Cref{UnimodularPatchworking}]
  Performing a base change $t \to t^d$ and applying the \Cref{lem:Refinements},
	we may assume that there exists a strongly unimodular subdivision 
	$\PPP$ of $\R^n$ such that 
	$\X = \{\sigma \in \PPP : \sigma \subset X\}$ and 
	$\Sigma \subset \widetilde{\Sigma} = \RecCone(\PPP)$.
	By \Cref{UnimodularPatchworking} there exists a homeomorphism
	$\R\widetilde{\Sigma} \approx \PW\T\widetilde{\Sigma}$ 
	which sends $\R \overline{\XX}_t$ to $\PW(\overline{X}, \EEE)$ 
	and respects torus orbits. 
	The latter property ensures that it restricts to homeomorphism
	$\R\Sigma \approx \PW\T\Sigma$ with the analogous property.
\end{proof}

In the remaining subsections, we will prove
\Cref{UnimodularPatchworking}. 
The general strategy will be as follows. 
A strictly unimodular subdivision 
$\PPP$ yields a semi-stable degeneration,
both on the classical and tropical side. 
Due to the nice properties of semi-stable degenerations, 
the topology of $( \R\Sigma, \R \overline{\XX}_t)$ and 
$(\PW\T\Sigma, \PW(\overline{X}, \EEE))$
can be described in terms 
of the classical and tropical special fibre, respectively. 
Then the last step is to relate the 
classical and tropical special fibres, which boils
down to a combinatorial statement about an isomorphism of posets.

\subsection{Regular CW complexes and pairs} \label{SecRegularCWComplexes}

In this subsection, we present a convenient framework for the description of the homeomorphisms
that we are about to construct. 

Let $X$ be a Hausdorff topological space. A subset $\sigma \subset X$
is called a \emph{(regular) $n$-cell}
if there exists an homeomorphism $\alpha \colon D^n \to \sigma$. 
Here, $D^n = \{x \in  \R^n : |x| \leq 1\}$ denotes the closed unit $n$-ball. 
Note that $\sigma$ is compact and hence closed in $X$. 
We define the \emph{relative interior} of $\sigma$ by 
$\sigma^\diamond = \alpha((\D^n)^\circ)$ and 
the \emph{relative boundary} by $d \sigma = \sigma \setminus \sigma^\diamond = \alpha(\partial D^n)$.
These notions are well-defined, since if given a second homeomorphism
$\alpha' \colon D^n \to \sigma$, the composition $\alpha' \circ \alpha^{-1}
\colon D^n \to D^n$ is a homeomorphism and hence preserves $\partial D^n = S^{n-1}$. 
Note also that $n$, called the dimension of $\sigma$, is well-defined
since $D^n$ and $D^m$ are homeomorphic if and only if $n=m$.

\begin{definition} \label{def:RegularComplex}
  Let $X$ be a Hausdorff topological space. 
  A \emph{finite regular CW complex $\XXX$} with support $X$ is a finite collection 
	of subsets of $X$ such that 
	\begin{enumerate}
		\item 
		  each $\sigma \in \XXX$ is a $n_\sigma$-cell of $X$ for some $n_\sigma \in \NN$, 
		\item 
		  the relative interiors $\sigma^\diamond$, $\sigma \in \XXX$, form a disjoint union of $X$,
		\item 
		  for each $\sigma$, the relative boundary $d \sigma$ is a union of cells of smaller dimension.
	\end{enumerate}
\end{definition} 

A few comments: Given $\XXX$ according to the definition, we can choose a homeomorphism
$\alpha_\sigma \colon D^{n_\sigma} \to \sigma \subset X$ for every $\sigma \in \XXX$. 
By \cite[Proposition A.2]{Hatcher} (note that condition (iii) in \emph{loc.~cit.}~is not needed/automatic since
we restrict to the finite case), the $\alpha_\sigma$ yield characteristic maps 
of a CW complex supporting $X$. Since each $\alpha_\sigma$ is a homeomorphism, 
this CW complex is regular according to the standard definition, see \cite[Chapter 3, Definition 1.1]{LundellWeingram}.
Moreover, choosing different $\alpha'_\sigma$, we obtain a strictly equivalent CW complex, see 
\cite[p.~7]{LundellWeingram}. This justifies our simplified definition and the naming. 
We also use $|\XXX| = X = \bigcup_{\sigma \in \XXX} \sigma$ to denote the support of $\XXX$. 
We consider $\XXX$ as a poset with the partial ordering given by inclusion of cells. 
We will sometimes refer to this a the \emph{cell poset} of $\X$ to emphasize that
we forget all geometric information and only consider the poset structure. 
The previous discussion also implies that regular CW complexes in the sense of our definition
admit barycentric subdivisions. For convenience, we express this in the following statement.

\begin{proposition} \label{propBarySubdivisionComplex}
  Let $\XXX$ be a regular CW complex. Then there exists a homeomorphism 
	$\beta \colon |O(\XXX)| \to |\XXX|$ from
	the simplicial complex $O(\XXX)$ (the order complex of the poset $\XXX$)
	to $\XXX$ such that the image of the simplex in $O(\XXX)$ labelled by the chain
	$\sigma_1 \subsetneq \dots \subsetneq \sigma_k$ lies in $\sigma_k^\diamond$. 
	Such a $\beta\colon |O(\XXX)| \to |\XXX|$ is called a \emph{barycentric subdivision} of $\XXX$. 
\end{proposition}

\begin{proof}
  This is essentially \cite[Theorem 1.7]{LundellWeingram}, but for later reference
	we give a sketch. Fix characteristic maps 
	$\alpha_\sigma \colon D^{n_\sigma} \to \sigma$ for every $\sigma \in \XXX$. 
	We proceed by induction on the $k$-skeleton $\XXX^{(k)}$ of $\XXX$. 
	The initial case $k=0$ is trivial. Assume that a barycentric subdivision $\beta^{(k)} \colon
	|O(\XXX^{(k)})| \to |\XXX^{(k})|$ has already been constructed. 
	Given a $k+1$-cell $\sigma$, pulling back via 
	$\alpha_\sigma \colon S^{k} \to \partial\sigma \subset |\XXX^{(k)}|$
	we obtain a triangulation on $S^k$. Coning over $0 \in D^{k+1}$ produces 
	a triangulation of $D^{k+1}$, and hence $\sigma$, labelled canonically by chains
	of the form $\sigma_1 \subsetneq \dots \subsetneq \sigma_l \subset \sigma$.
	Repeating this for every $k+1$-cell $\sigma$, we obtain the desired $\beta_{k+1}$.
\end{proof}

A topological pair $\tau \subset \sigma$ is called a standard pair of cells of codimension $k$
if there exists an homeomorphism $\alpha \colon D^n \to \sigma$ such
that $\alpha^{-1}(\tau) = \{x \in \D^n : x_1 = \dots = x_k = 0\}$.
In particular, $\sigma$ is an $n$-cell, $\tau$ is a $n-k$-cell, and
$d\tau = \tau \cap d\sigma$ as well as $\tau^\diamond = \tau \cap \sigma^\diamond$.

\begin{definition} \label{def:regpair }
  Let $X$ be a Hausdorff topological space. 
	A \emph{finite regular CW pair} of codimension $k$ is a tuple $(\XXX, Y)$ such that
	$\XXX$ is a finite regular CW complex and $Y \subset |\XXX|$ is a subset 
	such that for each $\sigma \in \XXX$ either $\sigma \cap Y = \emptyset$ or
	$\sigma \cap Y \subset \sigma$ is a standard pair of cells of 
	codimension $k$.
\end{definition} 

It follows easily from the definition that in this case 
\[
  \YYY = \{\sigma \cap Y : \sigma \in \XXX, \sigma \cap Y \neq \emptyset\}
\]
is a regular CW complex, $\dim \YYY = \dim \XXX - k$
and $\partial(\sigma \cap Y) = \partial \sigma \cap Y$ as well as 
$(\sigma \cap Y)^\diamond
= \sigma^\diamond \cap Y$. 
Note that $Y = |\YYY|$ is a closed subset of $|\XXX|$. 
It will be convenient to 
use the notation
\[
  \XXX_Y = \{ \sigma \in \XXX : \sigma \cap Y \neq \emptyset\}. 
\]
By definition, we have a canonical bijection between $\XXX_Y$ and $\YYY$.
Since we only consider \emph{finite} regular CW complexes and pairs, we will drop the finite from now on.

The notion of barycentric subdivisions extends to regular CW pairs in the following
sense. 

\begin{proposition} \label{proppropBarySubdivisionPair}
  Let $(\XXX, Y)$ be a regular CW pair with $X = |\XXX|$.  
	Then there exists a barycentric subdivision 
	$\beta \colon |O(\XXX)| \to X$ of $\XXX$
	such that the image of the simplicial subcomplex $O(\XXX_Y) \subset O(\XXX)$
	(given by the subposet $\XXX_Y \subset \XXX$) is equal to $Y$. 
	Moreover, in this case the restricted map $\beta_Y \colon |O(\XXX_Y)| \to Y$
	is a barycentric subdivision of $\YYY$ using the identification
	$O(\XXX_Y) = O(\YYY)$.
\end{proposition}

\begin{proof}
  We just have to slightly adapt the argument given for \Cref{propBarySubdivisionComplex}:
	We again choose characteristic maps $\alpha_\sigma \colon D^{n_\sigma} \to \sigma$ for every 
	$\sigma \in \XXX$, but additionally assume that 
	$\alpha_\sigma^{-1}(Y) = \{x \in \D^n : x_1 = \dots = x_k = 0\}$ whenever $\sigma \cap Y 
	\neq \emptyset$. Note that for any $m$ the pair $(\XXX^{(m)}, Y \cap |\XXX^{(m)}|)$
	is a regular CW pair as well. Hence the same recursive coning procedure as in 
	\Cref{propBarySubdivisionComplex} produces the desired homeomorphism. 
	Finally, given a chain $\sigma_1 \subsetneq \dots \subsetneq \sigma_l$ in $\XXX_Y$, 
	the image of the corresponding simplex lies in $\sigma_l^\diamond \cap Y
	= (\sigma_l \cap Y)^\diamond$, which shows that the restricted map
	$\beta_Y \colon |O(\XXX_Y)| \to Y$
	is a barycentric subdivision of $\YYY$.
\end{proof}

The following theorem states that, 
up to homeomorphisms, a regular CW complex $\XXX$ is determined by the poset structure 
on $\XXX$, and similarly, a regular CW pair is determined by the pair of posets 
$\XXX_Y \subset \XXX$. This follows easily from the existence of barycentric subdivisions.

\begin{thm} \label{prop:CWHomeomorphism}
  Let $(\XXX, Y)$ and $(\XXX', Y')$ be two regular CW pairs. 
	Assume that there exists poset isomorphism $\varphi : \XXX \to \XXX'$ that sends
	$\XXX_{Y}$ to $\XXX'_{Y'}$. Then there exists a homeomorphism 
	$\Phi \colon |\XXX| \to |\XXX'|$ that sends a cell $\sigma$ to the cell $\varphi(\sigma)$
	and sends $Y$ to $Y'$. 
\end{thm}

\begin{proof}
  The poset isomorphism $\varphi$ yields an induced simplicial homeomorphism
	$O(\varphi) \colon |O(\XXX)| \to |O(\XXX')|$ that sends the subcomplex
	$O(\XXX_Y)$ to the subcomplex $O(\XXX'_{Y'})$. Pre- and postcomposing 
	$O(\varphi)$ with barycentric subdivisions as provided in 
	\Cref{proppropBarySubdivisionPair} yields the desired homeomorphism $\Phi$. 
\end{proof}

\subsection{The local cases}\label{sec:local}

In this subsection we will deal with the local case of patchworking
which, for non-singular tropical subvarieties means the case of linear spaces:
We fix a real linear subvariety $\bL \subset (\C^*)^n$ of dimension $d$. 
Its locus of real points is $\R \bL \subset (\R^*)^n$.
We denote by $L = \Trop(\bL)$ the support of the associated  matroid fan
and by $\EEE= \EEE_\bL$ the associated real phase structure on $L$.
As usual, $\PW L = \PW(L, \EEE) \subset \PW \R^n \approx (\R^*)^n$ 
denotes the associated patchwork. 
It suffices to consider a single orthant, and by symmetry, we focus
on the positive orthant labelled by $\varepsilon = (0,\dots, 0) \in \Z_2^n$. 
We set $\Rpos = (0,\infty)$ and 
\begin{align*} 
  \Rpos \bL & := \R \bL \cap \Rpos^n && \subset \Rpos^n, \\
	\PWpos L &= \PW L \cap \R^n((0,\dots, 0)) && \subset \R^n.	
\end{align*}
Note that $\Rpos \bL$ is non-empty if and only if $\PWpos L$ is
non-empty by definition of $\EEE$.

Let $\Sigma$ be a complete unimodular fan.
The closure of $\PWpos L$ in $\T \Sigma$ is denoted by
$\PWnn L$. It is equal to 
$\PW \overline{L} \cap \T\Sigma((0,\dots, 0))$.
The closure of the positive orthant in $\R\Sigma$ is denoted 
by $\Rnn \Sigma$ (note that $\Rnn \Sigma$ can also be described
as the toric variety over the semigroup $\Rnn$, so the notation is consistent).
The closure of $\Rpos \bL$  in $\Rnn\Sigma$ 
is denoted by $\Rnn \bL$.
It is equal  to $\R \overline{\bL} \cap \Rnn\Sigma$.
The goal of this subsection is to prove that under suitable conditions
the pairs $\Rnn \bL \subset \Rnn\Sigma$ and 
	$\PWnn L \subset \T\Sigma$ are 
	standard pairs of cells  of dimension $(d,n)$.

Note that $L$ is a union of a collection of cones in $\Sigma$ if and only if
$\Sigma_L = \{\sigma \in \Sigma : \sigma \subset L\}$
is a polyhedral subdivision for $L$. We want to prove
the following two analogous statements, which we state separately for more
convenient referencing. 

\begin{thm} \label{thm:StandardPairsClassical}
  Assume that $\Rpos \bL$ is non-empty. 
	Assume that $L$ is a union of cones in $\Sigma$.
	Then $\Rnn \bL \subset \Rnn\Sigma$ is a 
	standard pair of cells  of dimension $(d,n)$.
\end{thm}

\begin{thm} \label{thm:StandardPairsTropical}
  Assume that $\PWpos L$ is non-empty. 
	Assume that $L$ is a union of cones in $\Sigma$.
	Then 
	$\PWnn L \subset \T\Sigma$ is a  
	standard pair of cells  of dimension $(d,n)$.
\end{thm}

We proceed in several steps.

\begin{lemma} \label{lem:ProjectiveCase}
  Assume that $\Rpos \bL$ (and hence $\PWpos L$) is non-empty.
	Then the pairs
	$\Rnn \bL \subset \Rnn\PP^n$
	and $\PWnn L \subset \T\PP^n$
	are standard pairs of cells
	of dimension $(d,n)$.
\end{lemma}

We emphasize that in this lemma we do not require that 
$L$ is a union of cones of $\Sigma$. 

\begin{proof}
  It is clear that $\Rnn\bL \subset \Rnn\PP^n$
	is a standard pair of cells:
	In projective coordinates, $\Rnn\PP^n$ can be described as the 
	standard  simplex $x_0, \dots, x_n \geq 0$, $x_0 + \cdots + x_n = 1$.
	Since $\Rnn \bL$ is the intersection 
	of this simplex with a linear subspace 
	which intersects its interior, the claim follows. 
		
	To finish, we construct a homeomorphism between
	$\Rnn \bL \subset \Rnn\PP^n$
	and $\PWnn L \subset \T\PP^n$.
	As explained above, we regard
	$\Rnn \bL \subset \Rnn\PP^n$ as
	the intersection of the simplex with a
	linear space. We consider a barycentric subdivision 
	of $\Rnn\PP^n$ where the barycenter of a face is chosen
	in $\Rnn\bL$ whenever the intersection
	with the relative interior of this face is non-empty. 
	We obtain a pair of simplicial 
	complexes with the following characteristics:
	\begin{itemize}
		\item 
			vertices of $\Rnn\PP^n$: non-empty subsets of $E = \{0, \dots, n\}$.
		\item 
			simplices of $\Rnn\PP^n$: chains under inclusion.
		\item 
			vertices of $\Rnn \bL$: 
			flats of $M_\bL$ that touch $\varepsilon = (0, \dots, 0)$, 
		  that is, $\FFF_{lv}^{\MMM_\bL}((0, \dots, 0))$ from 
			\Cref{ex:LasVergnas}.
		\item 
			simplices of $\Rnn \bL$: 
			chains of $\FFF_{lv}^{\MMM_\bL}((0, \dots, 0))$
		  under inclusion.
	\end{itemize}
	As  explained in \Cref{ex:LasVergnas},  $\PWnn L \subset \T\PP^n$
	is a pair of simplicial complexes with exactly the same poset structure, 
	which induces a simplicial homeomorphism between the pairs. 
\end{proof}

We now 
first deal with the tropical side, that is, with 
the pair $\PWnn L \subset \T\Sigma$.

\begin{lemma} \label{lem:TropicalClosures}
  For any complete unimodular fan $\Sigma$, 
	the space $\T \Sigma$ is a ball. 
	Moreover, given a subset $X \subset \R^n$ which is a union of
	cones in $\Sigma$, the pair $\overline X \subset \T \Sigma$ is homeomorphic
	to the pair of PL-stars at $0$ of $X \subset \R^n$. 
\end{lemma}

We use PL-stars only for simplicity, since they are well-defined up to PL-homeomorphism. 
We write $\pStar(X) = \pStar_0(X)$ for the PL-star at the origin $0  \in \R^n$.

\begin{proof}
  Given a cone $\sigma \in \Sigma$ of dimension $k$, choosing primitive generators $v_1, \dots, v_k$
	we obtain an identification with $[0,\infty)^k$.
	The closure $\overline{\sigma}$ in $\T \Sigma$ is then canonically homeomorphic
	to an infinite cube $[0,\infty]^k$.
	More generally, suppose  $X$ is the support of a subset of cones of $\Sigma$ and denote this subset by $\Sigma(X) \subset \Sigma$. Then
	the closure $\overline{X}$ in $\T \Sigma$ (or, equivalently, in $\T \Sigma(X)$)
	is obtained by gluing these infinite cubes $\overline{\sigma}$, $\sigma \in \Sigma(X)$,
	as a direct limit governed by the poset $\Sigma(X)$.
	
	Fix a homeomorphism $[0,\infty] \to [0,1]$ that sends $0$ to $0$.
	This induces an homeomorphism between $\overline{\sigma}$ and the (finite) cube $C_\sigma$
	in $\sigma$ whose vertices are the points $\sum_{i \in I} v_i$, $I \subset \{1, \dots, k\}$. 
	Moreover, these homeomorphisms glue to a homeomorphism between the direct limit $\overline{X}$
	and the set 
	\[
	  \pStar(X) = \bigcup_{\sigma \in \Sigma(X)} C_\sigma.
	\]
	As the notation suggests, the set $\pStar(X)$ is a PL-star for $X$ in $0$. 
	Finally, applying the same reasoning to $\R^n$ and $\Sigma(\R^n) = \Sigma$, 
	we obtain a homeomorphism of pairs between $\overline{X} \subset \overline{\R^n} = \T \Sigma$
	and $\pStar(X) \subset \pStar(\R^n)$. Of course, $\pStar(\R^n)$ is a ball.
\end{proof}

We can now finish the proof of \Cref{thm:StandardPairsTropical}.

\begin{proof}[Proof of \Cref{thm:StandardPairsTropical}]
  By \Cref{lem:ProjectiveCase}, we now that 
	$\PWnn {L} \subset \T\PP^n$ is 
	a standard pair of cells. 
	Since $0 \in \R^n$ corresponds to a point in the relative interior
	of this pair, it follows that 
	$\pStar(\PWpos L) \subset \pStar(\R^n)$ 
	is a standard pair of cells.
	By \Cref{lem:TropicalClosures}, this implies that 
	$\PWnn L \subset \T\Sigma$ is a standard pair of cells.
\end{proof}

We now turn our attention to the pair 
$\Rnn \bL \subset \Rnn\Sigma$. 

\begin{proof}[Proof of \Cref{thm:StandardPairsClassical}]
  As we mentioned in \Cref{rem:ExtendedLog}, 
	$\Rnn \Sigma$ is 
	homeomorphic to $\T \Sigma$ via the extended logarithm map $\Log$. 
	Hence by \Cref{lem:TropicalClosures},
	$\Rnn \Sigma$ is a ball.
	We now proceed by induction on $\dim(\bL)  = d$. 
	For $d = 0$, the statement is trivial. 
	We now prove the induction step $d-1 \to d$. 
	
	By our first remark, the subsets $\Rnn \Star_\rho \Sigma$, $\rho \in \Sigma$, 
	form a regular CW complex (with support $\Rnn \Sigma$). 
	By the induction assumption, we know that 
	for $\rho \subset L$ of positive dimension, the
	pair 
	$\Rnn \bL \cap \Rnn\Star_\rho \Sigma \subset \Rnn \Star_\rho \Sigma$
	is a standard pair of cells, while for $\rho \not\subset L$ the intersection 
	$\Rnn\bL \cap \Rnn \Star_\sigma \Sigma$ is empty. 
	It follows that  $\partial \Rnn \bL \subset \partial \Rnn \Sigma$ 
	is a regular CW pair.
	
	The analogous argument shows the corresponding tropical pair
	$\partial \PWnn L \subset \partial \T \Sigma$
	is a regular CW pair with the same poset structure
	(since $\Rnn\bL \cap \Rnn\Star_\rho \Sigma\neq \emptyset$
	if and only if $0 \in \EEE(\rho)$ if and only if
	$\PWnn L \cap \T\Star_\rho \Sigma \neq \emptyset$).
	Hence by \Cref{prop:CWHomeomorphism} the two pairs are homeomorphic. 
	It follows that both pairs are standard pairs of spheres, 
	since by \Cref{thm:StandardPairsTropical} 
	$\PWnn L \subset \T \Sigma$
	is a standard pair of cells. 
	
	Since $L$ is a union of cones in $\Sigma$, $\overline{\bL}$ is transversal 
	with respect to the toric boundary of $\C\Sigma$ in the following sense:
	For any $p \in \overline{\bL}$, the torus orbit containing $p$
	and $\overline{\bL}$ intersect transversally in $p$. 
	This follows from  \Cref{prop:SchoenPropertiesII}.
	Transversality with the toric boundary implies that there 
	exists a neighbourhood $\partial \Rnn\Sigma \subset U \subset \Rnn \Sigma$
	which is homeomorphic to $\partial \Rnn \Sigma \times [0,1]$ 
	and such that this homeomorphism identifies
	$\Rnn\bL \cap U$ with $\partial \Rnn\bL \times [0,1]$,
	see \cite[Theorem 6.1]{Rau-RealSemiStable}.
	If follows that we may replace 
	$\Sigma$ by any complete unimodular fan supporting $L$,
	since by the previous argument 
	$\partial \Rnn \bL \subset \partial \Rnn \Sigma$ is a standard pair of spheres
	for any such choice of $\Sigma$. 
	Hence we assume from now on that $\Sigma = \Pi^n$ is the permutahedral fan in $\R^n$. 
	
	The permutahedral variety $\R \Pi^n \to \RP^n$ is obtained from $\RP^n$ 
	from consecutively blowing up
	(strict transforms) of coordinate subspaces of $\RP^n$ starting with points, lines, etc. 
	We know from \Cref{lem:ProjectiveCase}
	that $\Rnn \bL \subset \Rnn\PP^n$ is a standard pair of cells, 
	so it remains to 
	prove that the topology does not change under the aforementioned blow-ups.
	For clarity, we denote the closures in $\Rnn\PP^n$ and $\Rnn\Pi^n$ 
	by $\Rnn^\PP \bL$ and $\Rnn^\Pi \bL$, respectively. 
	The local model for all these blow-ups is blowing up $\Rnn^k \times \R^{n-k}$ along
	$\{0\}^k \times \R^{n-k}$ 
	(the image of $\bL$ after this change of coordinates is still a linear subvariety
	in $(\C^*)^n$ which we call $\bL$ again by abuse of notation). 
	We may assume that $0 \in \Rnn^\PP \bL$. 
	In order to compare $\Rnn^\PP \bL$ and $\Rnn^\Pi \bL$
	it is sufficient to consider one slice $\Rnn^k \times \{0\}^{n-k}$,
	that is, we may further reduce to the situation of the blow-up $\Bl_0 \Rnn^n$ 
	of $\Rnn^n$ in the origin $0$. 
	By abuse of notation, we now denote by $\Rnn^\PP \bL$ and $\Rnn^\Pi \bL$
	their restriction to $\Rnn^n$ and $\Bl_0 \Rnn^n$.
	
	The blow-up can be constructed explicitly as 
	\[	  
	  \R_{\triangle}^n = \{x \in \Rnn^n : \sum_i x_i \geq 1\} 
	\]
	with blow-down map given by 
	$\pi \colon x \mapsto (1 - \frac{1}{\sum_i x_i}) x$. 
	Under this identification, $\Rnn^\Pi \bL \subset \Bl_0 \Rnn^n$
	is mapped to $\Rnn^\triangle \bL := \Rnn^\PP \bL \cap \R_{\triangle}^n \subset \R_{\triangle}^n$. 
	We now construct a homeomorphism $\varphi \colon \R_{\triangle}^n \to \Rnn^n$
	such that $\varphi(\Rnn^\triangle \bL) = \Rnn^\PP \bL$. 
	We fix a point $a \in \Rpos\bL$ such that $\sum_i a_i = 1$. 
	For any $I \subsetneq \{1, \dots, n\}$, we consider the simplices
	\begin{align*} 
		S_I = \Conv(a, 2a, e_i, i \in I), & & T_I = \Conv(0, 2a, e_i, i \in I).
	\end{align*}
	Then $\varphi$ is the piecewise linear map that maps
	each $S_I$ to $T_I$ by sending $a \mapsto 0$ and being identity
	on all other vertices as well as outside of the union of the $S_I$. 
	Clearly, $\varphi$ yields a homeomorphism with
	$\varphi(\Rnn^\triangle \bL) = \Rnn^\PP \bL$.
\end{proof}

\subsection{Semi-stable degenerations --- classical side}\label{subsec:classical} 

In this section we continue with the notation and set up of \Cref{UnimodularPatchworking}. 
We first recall the main elements and explain its connection to \cite{Rau-RealSemiStable}. 

Again let $\PPP$ be a strongly unimodular subdivision of $\R^n$.
As mentioned, the set of recession cones  $\Sigma = \RecCone(\PPP)$ is a pointed unimodular complete fan.
We set $\widetilde{\PPP} = C(\PPP)$.

Note that $\C \widetilde{\PPP}$ is a $(n+1)$-dimensional non-singular toric variety 
that comes together with a toric morphism  
$f \colon \C \widetilde{\PPP} \to \C$. 
The generic fibre $f^{-1}(t)$, $t\neq 0$, is the toric variety $\C \Sigma$. 
The orbits of the torus action on the special fibre 
$\C \widetilde{\PPP}_0 = f ^{-1}(0)$ are in canonical bijection with the faces $\sigma$ of $\PPP$. 
Their closures are the toric varieties $\C \Star_\sigma \PPP$.

We set
\[
  \C \PPP := f^{-1}(\D).
\]
By abuse of notation, we will use the same letter $f$ for the restriction
$f \colon \C\PPP \to \D$, and also write $\C\PPP_0 = \C \widetilde{\PPP}_0 = f ^{-1}(0)$.

Now let $\XX$ be an analytic family over $\D^*$. 
Note that $(\C^*)^n \times \D^* \subset \C\PPP = f^{-1}(\D)$. 
We denote by $\overline{\XX}$ the closure of $\XX$
in $\C\PPP$, and set $\XX_0 = \overline{\XX} \cap \C\PPP_0$.

\begin{proposition} \label{prop:SchoenPropertiesII}
  Under the assumptions of \Cref{UnimodularPatchworking}, 
	the following holds true. 
	\begin{enumerate}
		\item 
			For any $\sigma \in \PPP$, the scheme-theoretic intersection 
			$\overline{\XX}	\cap \C\OOO_{C(\sigma)}$ in $\C\PPP$ is 
			equal to $\XX_\sigma = \init_\sigma \XX / \Tbold_\sigma$
			(note that $\C\OOO_{C(\sigma)} \cong \C\OOO_{\sigma}$).
		\item 
		  The scheme $\overline{\XX}$ is non-singular and 
			for any $\sigma \in \PPP$ the intersection 
			$\overline{\XX}	\cap \C\OOO_{C(\sigma)}$ in $\C\PPP$ 
			is transversal. 
	\end{enumerate}
\end{proposition}

\begin{proof}
  By \Cref{prop:SchoenPropertiesI}, the subdivision $\X = \{\sigma \in \PPP : \sigma \subset X\}$
	is tropical. 
	Then, item (1) is for example contained in \cite[Proposition 2.4.4]{SpeyerThesis}.
	Item (2) follows since the intersection is non-singular and of expected dimension 
	by (1). 
\end{proof}

As before, given a polyhedron $\sigma$ in $\R^n$, we
use the shorthand
\[
  \Z_2(\sigma) = \Z_2^n / T_{\Z_2} (\sigma).
\]
Given the data of \Cref{UnimodularPatchworking}, we define the following poset:
\begin{equation} \label{eq:Qposet} 
  Q(\PPP) = \{ (\sigma, \varepsilon) : \sigma \in \PPP, \varepsilon \in \Z_2(\RecCone(\sigma)) \},
\end{equation}
cf.\ $Q(\Sigma)$ from \Cref{sec:tropicalPatchworking}. 
The partial order is given by 
\[
  (\sigma, \varepsilon) \leq (\tau, \delta) :\Longleftrightarrow 
	\tau \subset \sigma, \varepsilon = \pi(\delta),
\]
where $\pi \colon \Z_2(\RecCone(\tau)) \to \Z_2(\RecCone(\sigma))$ denotes
the canonical projection. We also consider the subposet
\begin{equation} \label{eq:QposetII} 
  Q(\X, \EEE) = \{(\sigma, \varepsilon) \in Q(\PPP) : \sigma \in \X,
					\varepsilon \in \pi(\EEE(\sigma)) \} \subset Q(\PPP).
\end{equation}
Here, $\pi \colon \Z_2^n \to \Z_2(\RecCone(\sigma))$ 
			denotes the canonical projection.

\begin{thm} \label{thm:ClassicalSide}
  Under the assumptions of \Cref{UnimodularPatchworking} and 
	for sufficiently small and positive $t \in \D^* \cap \R$,
	we can equip $\R\Sigma$ with the structure 
	of a regular CW complex such that
	\begin{enumerate}
		\item 
		  the poset $Q(\PPP)$ is isomorphic to the face poset of this CW complex
			via an assignment $(\sigma, \varepsilon) \mapsto C(\sigma, \varepsilon)$,
		\item 
		  the cell $C(\sigma, \varepsilon)$ is contained in $\R\Star_\rho \Sigma \subset \R\Sigma$ 
			if and only if $\rho \subset \RecCone(\sigma)$,
		\item 
		  the pair $\R \overline{\XX}_t \subset \R\Sigma$ is a regular CW pair with respect
			to this CW complex and the set of cells intersecting $\R \overline{\XX}_t$ corresponds
			to $Q(\X, \EEE)$, 
			\[
			  Q(\X, \EEE) = \{(\sigma, \varepsilon) \in Q(\PPP) : C(\sigma, \varepsilon) \cap \R \overline{\XX}_t \neq \emptyset\}.
			\]
	\end{enumerate}
\end{thm}

\begin{proof}[Proof of \Cref{thm:ClassicalSide}]
The map  $f \colon \C\PPP \to \D$ is a totally real 
	semi-stable degeneration, cf.\
	\cite[Lemma 9.1]{Rau-RealSemiStable}.
	It then follows by \cite[Theorem 1.1]{Rau-RealSemiStable}
	that the generic real fibre $\R\Sigma$ is homeomorphic
	to a certain real-oriented blow-up 
	$\Bl_+(\R\PPP_0)$ of the special real fibre $\R\PPP_0$.
	Instead of defining $\Bl_+(\R\PPP_0)$ carefully, we apply 
	\cite[Corollary 9.3]{Rau-RealSemiStable}
	which states that 
	$\Bl_+(\R\PPP_0)$ carries a stratification into closed
	subsets $C(\tau, \varepsilon)$ labelled by the elements $(\tau, \varepsilon) \in Q(\PPP)$
	whose inclusion relations are given by the partial order on $\PPP$. 
	More precisely, $C(\sigma, \varepsilon)$ is actually a copy 
	of $\R \Star_\sigma \PPP(\pi(\varepsilon))$,
	with $\pi \colon \Z_2(\RecCone(\sigma)) \to \Z_2(\sigma)$
	the canonical projection. 
	It follows that $C(\sigma, \varepsilon)$
	is an $(n-\dim \sigma)$-cell by \Cref{thm:StandardPairsClassical}. 
	The fact that these cells form a regular CW complex 
	satisfying the properties stated in items (1) and (2) 
	is part of \cite[Corollary 9.3]{Rau-RealSemiStable}.
	
	Let us now proceed to item (3). 
	It follows from \Cref{prop:SchoenPropertiesII} that 
	that $\XX$ together with the non-special toric boundary divisors in
	$\C\PPP$ (labelled by the rays of $\Sigma$) form a transversal 
	collection of submanifolds in the sense of 
	\cite[Section 6]{Rau-RealSemiStable}. 
	Hence by \cite[Theorem 6.1]{Rau-RealSemiStable}, 
	for sufficiently small positive $t$, the CW structure from 
	the previous paragraph can be chosen such that 
		  the pair $\R \overline{\XX}_t \cap C(\sigma, \varepsilon)  \subset C(\sigma, \varepsilon)$
			is homeomorphic to 
			$\overline{\XX}_0 \cap \R\Star_\sigma \PPP(\varepsilon) \PPP \subset \R\Star_\sigma \PPP(\varepsilon)$.
	By \Cref{prop:SchoenPropertiesII} 
	we have $\overline{\XX}_0 \cap \C \Star_\sigma \PPP = \overline{\XX}_\sigma$.
	Hence the pair in the last item is equal to 
	$\R\overline{\XX}_\sigma ( \varepsilon)\subset \R\Star_\sigma \PPP(\varepsilon)$.

	By definition, $\R\XX_\sigma (\varepsilon)$ is non-empty if and only if 
	$(\sigma, \varepsilon) \in Q(\X, \EEE)$.
	If non-empty, $\XX_\sigma$ is linear up to a coordinate change in 
	$\GL(n,\Z)$ and $\Trop(\XX_\sigma) = |\Star_\sigma \X|$ is a union of cones
	in $\Star_\sigma \PPP$. Hence, by \Cref{thm:StandardPairsClassical} 
	the pair 
	$\R\overline{\XX}_\sigma ( \varepsilon)\subset \R\Star_\sigma \PPP(\varepsilon)$
	is a standard pair
	of cells. 
	This proves the statement. 
\end{proof}

We note that the homeomorphisms from \cite{Rau-RealSemiStable} used above
are relative and stratified versions of similar results
in the literature: 
In \cite[Section 3 and 4]{Par-BlowAnalyticRetraction}, similar results are proven
in the context retracting a family to the special fibre. 
More general statements in the setting of log geometry and the Kato-Nakayama 
space appear in \cite[Theorem 5.1]{NO-RelativeRoundingToric} and 
\cite[Proposition 6.4]{Arg-RealLociLog}.

\subsection{Semi-stable degenerations --- tropical side}\label{subsec:tropical}

In this subsection, we consider the pair of spaces $(\PW\T\Sigma, \PW(\overline{X}, \EEE))$  
arising from patchworking $\overline{X}$ with a
real phase structure $\EEE$  
from Section \ref{sec:tropicalPatchworking}. 
Even though this pair comes with natural CW structures from the beginning, 
it will be convenient to study it via 
the tropical analogue of the the semi-stable degeneration $\C\PPP$ used in \Cref{subsec:classical}.
We will use this approach to construct 
a tropical \emph{special fibre} $(\T \PPP_{\infty},\overline{X}_\infty)$.  
We will demonstrate a homeomorphism of the pairs 
$(\T \PPP_{\infty}, \overline{X}_\infty)$ and $(\T \Sigma, \overline{X})$, 
see Proposition \ref{prop:TropicalSpecialFibre}.
This produces a homeomorphism of the pairs 
$(\PW\T\Sigma, \PW(\overline{X}, \EEE))$ 
and $(\PW \T \PPP^+_\infty, \PW(\overline{X},\EEE)^+_\infty)$. 
We then finish the proof of 
\Cref{UnimodularPatchworking} by showing that
$(\PW \T \PPP^+_\infty, \PW(\overline{X},\EEE)^+_\infty)$, 
and hence $(\PW\T\Sigma, \PW(\overline{X}, \EEE))$, 
is a regular CW pair whose poset structure
is the same as the one described in \Cref{thm:ClassicalSide}. 

We fix a strongly unimodular subdivision $\PPP$
of $\R^n$ and denote by $\Sigma = \RecCone(\PPP)$ its pointed 
unimodular recession fan. 
We denote by $\widetilde{\PPP} =C(\PPP)$ 
the pointed unimodular fan with support $\R^n \times \R_{\geq}$ 
described in \Cref{subsec:classical}.
The associated $(n+1)$-dimensional tropical toric variety $\T {\PPP} := \T \widetilde{\PPP}$ 
comes with a canonical projection to $\T$. The generic fibre $\T\PPP_t$, $t \in \R$, is a copy of the toric variety $\T\Sigma$. 
The special fibre over $-\infty$, denoted by $\T {\PPP}_{\infty}$,
is covered by the tropical toric varieties associated to the stars of all faces in $\PPP$,
\begin{equation} \label{eq:tropicalspecialfibre}
  \T {\PPP}_{\infty} = \bigcup_{\tau \in \PPP} \T \Star_\tau \PPP.
\end{equation}
Note that by \Cref{lem:TropicalClosures} this decomposition
is a regular CW complex whose face poset, via the given identification,
is isomorphic to $\PPP^\text{op}$ ($\PPP$ with reversed partial order).

Given a subcomplex $\X \subset \PPP$,
we denote by 
$\overline{\X}$ the closure of $\X$ 
as polyhedral complex in $\T \Sigma$ described in \Cref{eq:CanonicalComp}.
The underlying sets are denoted by $X$ and $\overline{X}$.
Similarly, we denote by $C(\X) = \{C(\sigma) : \sigma \in \X\}$
the associated subcomplex of $\widetilde{\PPP}$, by $C(X)$ its support and by $\overline{C(\X)}$ the closure of $C(\X) $ as a polyhedral complex in $\T {\tilde{\PPP}}$.

\begin{definition}\label{def:tropicalspecialfibre}
The support of the polyhedral complex
\[
  \overline{\X}_\infty = \overline{C(\X)} \cap \T {\PPP}_{\infty}.
\]
is the \emph{tropical special fibre} 
of $X$ associated to $\X$. We denote it by $\overline{X}_\infty$. 
\end{definition}

\begin{example}\label{ex:1dim}
Let us consider the small example of the subdivision $\PPP$ of $\R$ 
with vertices $\{0, \pm 1\}$ and $1$-cells $(-\infty,-1] \cup [-1,0] \cup [0,1] \cup [1,+\infty)$. 
In \Cref{fig:1dimDeg} we draw the fan $\tilde{P} = C(\PPP)$ on the left hand side. If we set $X=\left\lbrace -1 , 0 ,  1 \right\rbrace$, then $\overline{C(X)}$ consists of the three rays of $\tilde{\PPP}$. 
 
 On the right hand side of \Cref{fig:1dimDeg} is the tropical toric variety  $\T {\tilde{\PPP}}$. The special fibre 
 $\T {\PPP}_{\infty}$ is depicted in red. Moreover, the special fiber $\overline{X}_\infty$ is the contained in the closure of $C(X) \in \T {\tilde{\PPP}}$ and consists of the three points in blue.

\begin{figure}
\begin{tikzpicture}
\fill[color=gray!20] (-2.5,-1) -- (-2.5,1)--(2.5,1)--(2.5,-1);
    \draw[thick, -] (2.5,-1) -- (-2.5,-1) ;
    \draw (0,-1) node{$\bullet$};
    \draw (0,-1.3) node{\tiny{$(0,0)$}};
    \draw[thick, -] (0,-1) -- (-2,1) ;
     \draw[thick, -] (0,-1) -- (0,1) ;
    \draw[thick, -] (0,-1) -- (2,1) ;
    \draw (-1,0) node{$\bullet$};
    \draw (-0.7,0.3) node{\tiny{$(-1,1)$}};
    \draw (1,0) node{$\bullet$};
    \draw (0.7,0.3) node{\tiny{$(1,1)$}};
    \draw[thick, -] (-2.5,0) -- (2.5,0) ;
    \draw (-2.8,0) node{\tiny{$\PPP$}};  
    \draw (0,0) node{$\bullet$}; 
    
    \fill[color=gray!20] (5,-1)--(5,0) -- (5,1)--(6,2)--(8, 2)--(9,1)--(9,0)--(9,-1);
    \draw[thick, -] (5,-1) -- (5,1) ;
    \draw[thick, -, color=red] (5,1) -- (6,2) ;
    \draw[thick, -, color=red] (6,2) -- (7,2) ;
    \draw[thick, -,  color=red] (7,2) -- (8,2) ;
    \draw[thick, -,  color=red] (8,2) -- (9,1) ;
    \draw[thick, -] (9,1) -- (9,-1) ;
    \draw[thick, -] (7,0) -- (7,2);
        \draw[thick, -] (7,0) -- (5.5,1.5);
            \draw[thick, -] (7,0) -- (8.5,1.5);
    \draw[color=blue] (5.5,1.5) node{$\bullet$};
    \draw[color=blue] (7,2) node{$\bullet$};
    \draw[color=blue] (8.5,1.5) node{$\bullet$};
\end{tikzpicture}
\caption{On the left hand side is the  fan $C(\PPP)$ from \Cref{ex:1dim}. The polyhedral subdivision $\PPP$ is drawn at height one of $\R^2$ and the subcomplex $X$ in $\PPP$ consists of the three points.  On the right is the tropical toric variety  $\T {\tilde{\PPP}}$ with the special fiber $\T {\PPP}_{\infty}$ drawn in red. The special fiber  $\overline{X}_\infty$ is contained in $\T {\PPP}_{\infty}$ and consists of the three points drawn in blue. 
} \label{fig:1dimDeg}
\end{figure}
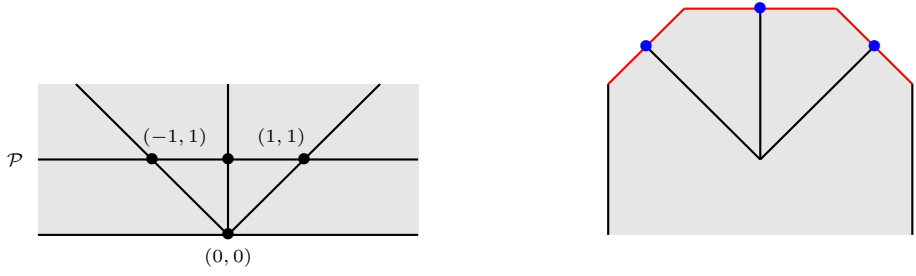
\end{example}

\begin{proposition} \label{prop:PolyhedralCWComplex}
  The polyhedral complexes $\overline{\X}$ and $\overline{\X}_\infty$ are regular CW complexes
	with supports $\overline{X}$ and $\overline{X}_\infty$, respectively. 
\end{proposition}

This is an immediate consequence of the following lemma.

\begin{lemma} \label{lem:FaceIsBall}
  Let $\sigma \subset \R^n$ be a strictly unimodular polyhedron containing a vertex. 
	Set $\rho = \RecCone(\sigma)$. Then $\overline{\sigma} \subset \T U_\rho$ is a 
	$\dim(\sigma)$-cell whose relative boundary $d \overline{\sigma}$
	is equal to the union of proper faces of 
	$\overline{\sigma}$.
\end{lemma}

\begin{proof}
Since $\sigma$ is a strictly unimodular polyhedron containing at least one vertex, we have $\sigma = \kappa + \rho$, where $\kappa $ is a bounded simplex such that 
$T(\kappa)$ and $T(\rho)$ have trivial intersection $\{0\}$. 
 Moreover, 	the closure of $\sigma$ satisfies $\overline{\sigma} = \kappa + \overline{\rho}$. 
 Since $\rho$ is homeomorphic to $[0,\infty)^k$, and $\overline{\rho} $ is homeomorphic to $ [0,\infty]^k$, we have that $\overline{\sigma}$ is homeomorphic to $\kappa \times [0,\infty]^k$, which proves the claim.
\end{proof}

Next, we give a description of the combinatorial structure of $\overline{\X}_\infty$
in terms of $\X$. 
Given any poset $P$, we denote by $\Int(P)$ the poset of (closed) intervals
in $P$. Here, an interval is a subset of $P$ of the form
$[a,b] = \{c \in P: a \leq c \leq b\}$ for $a \leq b \in P$. 
The partial order on $\Int(P)$ is given by inclusion,  
meaning 
$$[a,b] \leq [c, d] \Longleftrightarrow  c \leq a , b \leq d.$$

\begin{lemma} \label{lem:PosetSpecial}
  The cell poset $\overline{\X}_\infty$
	is isomorphic to $\Int(\X)$. 
\end{lemma}

\begin{proof}
	By definition, the faces of $\overline{\X}_\infty$ are the non-empty sets among
	\[
	  \overline{C(\sigma)} \cap \T\Star_\tau \PPP
	\]
	for $\sigma \in \overline{\X}_\infty$, $\tau \in \PPP$. 
	But $\overline{C(\sigma)} \cap \T\Star_\tau \PPP$
	is non-empty if and only if $C(\sigma) \cap \relint(C(\tau))$
	is non-empty (\cite[Proposition 3.2.4 (d)]{MikRau}) if and only if $\tau \subset \sigma$. 
	Hence the claim follows. 
\end{proof}

\begin{remark} \label{remX=P}
  Applying the construction to $\X = \P$, we obtain a polyhedral complex
	$\overline{\X}_\infty$ whose support is $\T\PPP_\infty$. 
	By \Cref{lem:FaceIsBall} and \Cref{lem:PosetSpecial} $\overline{\X}_\infty$ is a regular
	CW complex isomorphic as a poset to $\Int(\PPP)$. 
	On the other hand, we already  equipped  $\T\PPP_\infty$ with a regular
	CW structures given by \Cref{eq:tropicalspecialfibre} whose cell poset
	is isomorphic to $\PPP^\text{op}$. In fact, it is easy to see
	that $\overline{\X}_\infty$ is a cubical subdivision of $\T\PPP_\infty$
	which, on the level of cell posets, corresponds to passing from
	$\PPP^\text{op}$ to $\Int(\PPP^\text{op})$. Notice that the poset $\Int(\PPP^\text{op})$ is in fact equal to $\Int(\PPP)$. 	The following example illustrates this process. For clarity,
	we write $\T\PPP_\infty$ for the CW complex from \Cref{eq:tropicalspecialfibre}
	and $\T\PPP_\infty^\text{cub}$ for the CW complex $\overline{\X}_\infty$ when $\X = \PPP$. 
\end{remark}

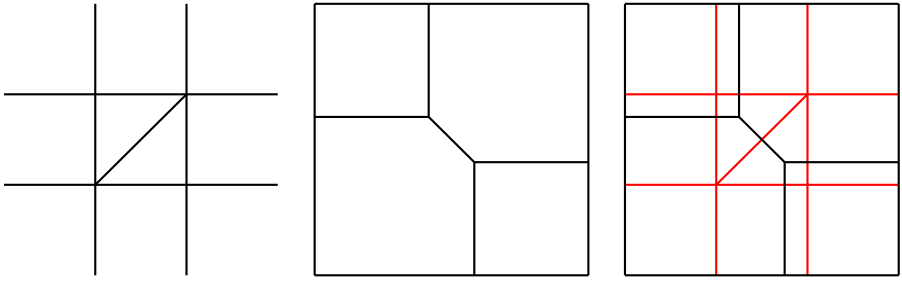
\begin{figure}

\begin{tikzpicture}[scale=0.6]
\coordinate (O) at (0,0);
\coordinate (P) at (2,2) ;
    \draw[thick, -] (4,0) -- (O) ;
    \draw[thick, -] (O) -- (-2,0) ;
    \draw[thick, -] (O) -- (0,-2) ;
       \draw[thick, -] (0,4) -- (O) ;
    \draw[thick, -] (O) -- (P) ;
    \draw[thick, -] (2,4) -- (P);
      \draw[thick, -] (P) -- (2,-2);
    \draw[thick, -] (4,2) -- (P) ;
     \draw[thick, -] (P) -- (-2,2) ;
    \end{tikzpicture}
\hspace{0.2cm}
\begin{tikzpicture}[scale=0.6]
      
    \draw[thick, -] (-2,1.5) -- (0.5,1.5) ;
    \draw[thick, -] (0.5,1.5) -- (0.5,4) ;
    \draw[thick, -] (0.5,1.5) -- (1.5,0.5) ;
     \draw[thick, -] (1.5,0.5) -- (4,0.5) ;
    \draw[thick, -] (1.5,0.5) -- (1.5,-2) ;
    
    \draw[thick, -] (-2,-2) -- (4,-2) ;
    \draw[thick, -] (4,-2) -- (4,4) ;
    \draw[thick, -] (-2,-2) -- (-2,4) ;
    \draw[thick, -] (-2,4) -- (4,4) ;
    \end{tikzpicture}
\hspace{0.2cm}
\begin{tikzpicture}[scale=0.6]
\coordinate (OO) at (7,0);
\coordinate (PP) at (9,2) ;
    \draw[thick, -, red] (11,0) -- (OO) ;
     \draw[thick, -, red] (OO) -- (5,0) ;
    \draw[thick, -, red] (7,4) -- (OO) ;
    \draw[thick, -, red] (OO) -- (7,-2) ;
    \draw[thick, -, red] (OO) -- (PP) ;
    \draw[thick, -, red] (9,4) -- (PP);
    \draw[thick, -,  red] (PP) -- (9,-2);
    \draw[thick, -, red] (11,2) -- (PP) ;
        \draw[thick, -, red] (PP) -- (5,2) ;
    
    \draw[thick, -] (5,-2) -- (11,-2) ;
    \draw[thick, -] (11,-2) -- (11,4) ;
    \draw[thick, -] (5,4) -- (11,4) ;
    \draw[thick, -] (5,-2) -- (5,4) ;
    
    \draw[thick, -] (5,1.5) -- (7.5,1.5) ;
    \draw[thick, -] (7.5,1.5) -- (7.5,4) ;
    \draw[thick, -] (7.5,1.5) -- (8.5,0.5) ;
     \draw[thick, -] (8.5,0.5) -- (11,0.5) ;
    \draw[thick, -] (8.5,0.5) -- (8.5,-2) ;   

\end{tikzpicture}
\caption{On the left is a strongly unimodular subdivision $\R^2$. In the middle the corresponding tropical special fibre, and on the right its cubical subdivision (the new cells of dimension less than $1$ are colored in red). } \label{fig:tropicalspecialfibre}
\end{figure}

\begin{example}\label{ex:ppp}
In Figure \ref{fig:tropicalspecialfibre}, we illustrate Remark \ref{remX=P} for a strongly unimodular subdivision  $\PPP$ of $\R^2$ (left-hand side of Figure \ref{fig:tropicalspecialfibre}). The tropical special fibre   $\T \PPP_{\infty}$ from \Cref{eq:tropicalspecialfibre} is depicted on the middle of the figure. On the right hand-side is depicted $\PPP_\infty$, called also $\T \PPP_{\infty}^{cub}$. As a cell complex, it is isomorphic to $\Int(\PPP)$.
\end{example}

To see that $(\T \PPP_{\infty}, \overline{X}_\infty)$ and $(\T \Sigma, \overline{X})$
are homeomorphic pairs, we will interpret $\overline{\X}_\infty$
as a certain type of cubical subdivision of $\overline{\X}$, which 
is explained next. 

A polyhedron $\kappa \in \R^n$ is \emph{bounded} if $\RecCone(\kappa) = \{0\}$. 
The subcomplex of bounded faces of $\X$ is denoted by 
\[
  \Bounded(\X) = \{\kappa \in \X : \RecCone(\kappa) = \{0\}\}.
\]
Notice that $\Bounded(\X)$ can be viewed  as a subcomplex of both $\X$ and $\overline{\X}$.
We will now construct polyhedral complexes refining $\X$ and $\overline{\X}$ 
by first taking a cubical subdivision of  $\Bounded(\PPP)$ 
and then extending this subdivision in a suitable way to infinity. 
We consider the subposet of $\Int(\Bounded(\X)) \times \Sigma$
given by 
\begin{equation} \label{eq:nextposet}
  P(\X) := \{([\lambda, \kappa], \rho) \in \Int(\Bounded(\X)) \times \Sigma : 
	  \kappa + \rho \in \X\}.   
\end{equation}

\begin{definition} \label{def:BoundedCubicalSubdivision}
  For any $\gamma \in \Bounded(\PPP)$, choose a point $p_\gamma$ in the 
	relative interior of $\gamma$. The \emph{bounded-cubical subdivision of $\X$}
	is the collection of (not necessarily rational) polyhedra in $\R^n$ given by
	\begin{align*} 
		\BC(\X) &:= \{F([\lambda, \kappa], \rho) : ([\lambda, \kappa], \rho) \in P(\X) \}, & \\
		F([\lambda, \kappa], \rho) &:= \Conv\{p_\gamma : \lambda \subset \gamma \subset \kappa\} + \rho 
			 & \subset  \R^n.
	\end{align*}
	Here, $\Conv$ denotes the convex hull. 
	The \emph{bounded-cubical subdivision of $\overline{\X}$}
	is the closure of $\BC(\X)$ in $\T\Sigma$, 
	\[
	  \BC(\overline{\X}) = \overline{\BC(\X)}. 
	\]
\end{definition}

We note that since we are only interested in topological properties, the 
non-rationality of $\BC(\X)$ does not present a problem. 
Alternatively, under the assumption that $\X$ is $\QQ$-rational, which is satisfied
in all our applications, the points $p_\sigma$ can be chosen in $\QQ^n$ which 
makes $\BC(\X)$ rational.

\begin{example}\label{ex:bcppp}
In Figure \ref{fig:boundedcc}, we consider on the left the same subdivision as in Example \ref{ex:ppp}. On the rest of the figure, we illustrate Definition \ref{def:BoundedCubicalSubdivision}. In the middle part, we choose a point (in red) for each bounded cell. On the right hand side is depicted $\overline{\BC(\PPP)}$. The right picture is the same as the one in Figure \ref{fig:boundedcc}: this is an illustration of Proposition \ref{prop:TropicalSpecialFibre}.

\begin{figure}
\begin{tikzpicture}[scale=0.6]
\coordinate (O) at (0,0);
\coordinate (P) at (2,2) ;
    \draw[thick, -] (4,0) -- (O) ;
    \draw[thick, -] (O) -- (-2,0) ;
    \draw[thick, -] (O) -- (0,-2) ;
       \draw[thick, -] (0,4) -- (O) ;
    \draw[thick, -] (O) -- (P) ;
    \draw[thick, -] (2,4) -- (P);
      \draw[thick, -] (P) -- (2,-2);
    \draw[thick, -] (4,2) -- (P) ;
     \draw[thick, -] (P) -- (-2,2) ;
    \end{tikzpicture}
\begin{tikzpicture}[scale=0.6]
\coordinate (OO) at (7,0);
\coordinate (PP) at (9,2) ;
    \draw[thick, -] (11,0) -- (OO) ;
     \draw[thick, -] (OO) -- (5,0) ;
    \draw[thick, -] (7,4) -- (OO) ;
    \draw[thick, -] (OO) -- (7,-2) ;
    \draw[thick, -] (OO) -- (PP) ;
    \draw[thick, -] (9,4) -- (PP);
    \draw[thick, -] (PP) -- (9,-2);
    \draw[thick, -] (11,2) -- (PP) ;
        \draw[thick, -] (PP) -- (5,2) ;
    
    \filldraw [red] (7,2) circle (2pt);
    \filldraw [red] (9,2) circle (2pt);
    \filldraw [red] (7,0) circle (2pt);
    \filldraw [red] (9,0) circle (2pt);

    \filldraw [red] (7.5,1.5) circle (2pt);
    \filldraw [red] (7,1.5) circle (2pt);
    \filldraw [red] (7.5,2) circle (2pt);
    \filldraw [red] (8.5,0.5) circle (2pt);
    \filldraw [red] (9,0.5) circle (2pt);
    \filldraw [red] (8.5,0) circle (2pt);
    \filldraw [red] (8,1) circle (2pt);
\end{tikzpicture}
\begin{tikzpicture}[scale=0.6]
\coordinate (OO) at (7,0);
\coordinate (PP) at (9,2) ;
    \draw[thick, -] (11,0) -- (OO) ;
     \draw[thick, -] (OO) -- (5,0) ;
    \draw[thick, -] (7,4) -- (OO) ;
    \draw[thick, -] (OO) -- (7,-2) ;
    \draw[thick, -] (OO) -- (PP) ;
    \draw[thick, -] (9,4) -- (PP);
    \draw[thick, -] (PP) -- (9,-2);
    \draw[thick, -] (11,2) -- (PP) ;
        \draw[thick, -] (PP) -- (5,2) ;

    \draw[thick, -,red] (5,1.5) -- (7.5,1.5) ;
    \draw[thick, -,red] (7.5,1.5) -- (7.5,4) ;
    \draw[thick, -,red] (7.5,1.5) -- (8.5,0.5) ;
     \draw[thick, -,red] (8.5,0.5) -- (11,0.5) ;
    \draw[thick, -,red] (8.5,0.5) -- (8.5,-2) ;   
\end{tikzpicture}
\caption{On the left the same subdivision as in Figure \ref{fig:tropicalspecialfibre}, on the middle in red one point for each bounded cell, and on the right the bounded cubical subdivision.}\label{fig:boundedcc}
\end{figure}
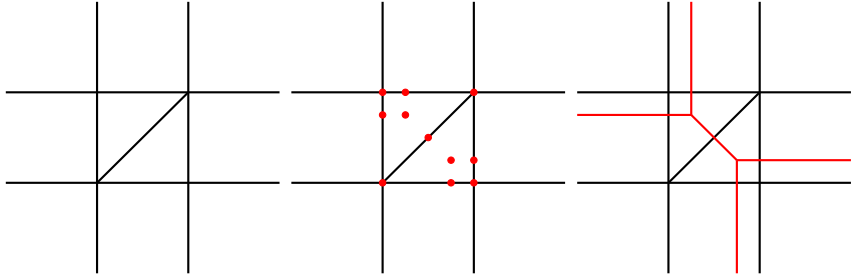

\end{example}

\begin{lemma} \label{prop:BoundedCubicalSubdivision}
  The collection of polyhedra $\BC(\X)$ is a polyhedral complex in $\R^n$ 
	refining $\X$ and such that
	$\RecCone(\BC(\X)) = \RecCone(\X)$. 
	The collection of polyhedra $\BC(\overline{\X})$
	is a polyhedral complex refining $\overline{\X}$.
	Moreover, $\BC(\overline{\X})$ is a regular CW complex. 
\end{lemma} 

\begin{proof}
The second phrase follows from the first. 
	In the first  statement, the claim about recession fans and the fact that the cones of 
	$\BC(\X)$ cover $X$ is obvious by construction. 
	We focus on showing that  $\BC(\X)$ is a polyhedral complex. 
	
	For $\lambda \subset \kappa \in \Bounded(\X)$, we set 
	$F(\lambda, \kappa) := F([\lambda, \kappa], 0)$. 
	By definition, these
	faces form a cubical subdivision of $\Bounded(\X)$. 
	The remaining faces are of the form $F(\lambda, \kappa) + \rho$
	for $\kappa + \rho \in \X$. Since $\kappa + \rho$ is strictly unimodular, 
	it follows that any face of $F(\lambda, \kappa) + \rho$
	is of the form $F(\lambda', \kappa') + \rho'$ where 
	$\lambda \subset \lambda' \subset \kappa' \subset \kappa$ and $\rho' \subset \rho$.
	Again since $\kappa + \rho$ is strictly unimodular, it is clear that $\kappa' + \rho' \in \X$, 
	hence $\BC(\X)$ is closed under taking faces. 
	
	Let now $F([\lambda, \kappa], \rho)$ and $F([\lambda', \kappa'], \rho')$ be two faces of 
	$\BC(\X)$ with non-empty intersection. Since $F([\lambda, \kappa], \rho) \subset \kappa + \rho$
	and $F([\lambda', \kappa'], \rho') \subset \kappa' + \rho'$, the intersection
	$(\kappa + \rho) \cap (\kappa' + \rho')$ is a non-empty face of $\X$. 
	Since $\X$ is a strictly unimodular,	
	$(\kappa + \rho) \cap (\kappa' + \rho')=(\kappa \cap \kappa') + (\rho \cap \rho')
	= \kappa'' + \rho''$, where $\kappa'' = \kappa \cap \kappa'$ and $\rho'' = \rho \cap \rho'$. 
	Since $F([\lambda, \kappa], \rho)$ and $F([\lambda', \kappa'], \rho')$
	intersect with $\kappa'' + \rho''$, we conclude 
	$\lambda, \lambda' \subset \kappa''$
	and 
  \begin{align*} 
	  F([\lambda, \kappa], \rho) \cap F([\lambda', \kappa'], \rho') 
		  &= F([\lambda, \kappa''], \rho) \cap F([\lambda', \kappa''], \rho') \\
			&= F([\lambda, \kappa''], \rho'') \cap F([\lambda', \kappa''], \rho'') \\
			&= F([\lambda'', \kappa''], \rho''),
	\end{align*}
	where $\lambda''$ is the smallest face of $\kappa''$ containing $\lambda$ and $\lambda'$. 
	
	For the third statement, we use a simple variation of
	\Cref{lem:FaceIsBall} and conclude that each of the polyhedra in  $\BC(\X)$ are homeomorphic to balls.  
\end{proof}

The relationship between $\BC(\overline{\PPP})$ and $\T {\PPP}_{\infty}^\text{cub}$
is as follows. 

\begin{proposition} \label{prop:TropicalSpecialFibre}
  There exists a canonical isomorphism of regular CW complexes from
	$\BC(\overline{\PPP})$ to $\T {\PPP}_{\infty}^\text{cub}$ 
	which respects toric strata and such that for any $\X \subset \PPP$
	the subcomplex $\BC(\overline{\X})$ is mapped to the subcomplex $\overline{\X}_\infty$. 
	In particular, 
	the pairs $(\T \Sigma, \overline{X})$ and $(\T {\PPP}_{\infty}, \overline{X}_\infty)$
	are homeomorphic. 
\end{proposition}

\begin{proof}
  We have already established in \Cref{prop:PolyhedralCWComplex} and 
	\Cref{prop:BoundedCubicalSubdivision} that 
	$\BC(\overline{\PPP})$ and $\T {\PPP}_{\infty}^\text{cub}$ are regular CW complexes. 
	In order to constuct a CW homeomorphism, by \Cref{prop:CWHomeomorphism}
	it is sufficient to construct an isomorphism between their cell posets.
	We use the short-hand $B = \Bounded(\PPP)$.
	
	First we note that since $\X$ is strictly unimodular, 
	the poset $\X$ is isomorphic to the subposet
	of $B \times \Sigma$ consisting of elements $(\kappa, \rho)$
	such that $\kappa + \rho \in \X$. 
	For the same reasons, the poset $\Int(\X)$ is isomorphic to 
	the subposet $P'(\X)$ 
	of $\Int(B) \times \Int(\Sigma)$ consisting of elements $([\lambda,\kappa], [\eta, \rho])$
	such that $\kappa + \rho \in \X$, via the map
	$([\lambda,\kappa], [\eta, \rho]) \mapsto (\lambda + \eta, \kappa + \rho)$.
	
	By \Cref{def:BoundedCubicalSubdivision}, the face poset of $\BC(\X)$ 
	is isomorphic to $P(\X) \subset \Int(B) \times \Sigma$ from \Cref{eq:nextposet}.
	By \Cref{prop:BoundedCubicalSubdivision} the recession fan 
	of $\BC(\X)$ is $\Sigma$. 
	Hence by \Cref{eq:PosetCompactification}
	the face poset of $\BC(\overline{\X})$ is isomorphic 
	to $P'(\X) = \Int(B) \times \Int(\Sigma)$.
	By the previous remark $P'(\X)$ is isomorphic to 
	$\Int(\X)$ which by \Cref{lem:PosetSpecial}
	is isomorphic to the cell poset of $\overline{\X}_\infty$. 
	
	Taking the closure of the bounded cubical subdivision for  $\X = \PPP$ 
	yields a CW homeomorphism between $\BC(\overline{\PPP})$ and $\T {\PPP}_{\infty}^\text{cub}$
	which clearly has the desired properties. 
\end{proof}

We now return to the patchwork setup. 
Recall that the patchwork of the tropical toric variety $\T \PPP$ is defined by 
\[
\PW \T \PPP:=  \bigsqcup_{\varepsilon \in \Z_2^{n+1}} \T \PPP({\varepsilon})/ \sim.  
\]
We have a canonical map (toric morphism) $\PW \T \PPP \to \PW \T$ where the latter space is just a union of two copies of 
$\T$ glued along $-\infty$. 
We denote the special fibre by $\PW \T \PPP_\infty$. 
Given $(\sigma, \varepsilon)$, $\sigma \in \PPP, \varepsilon \in \Z_2(\sigma)$, 
the assignment $(\sigma, \varepsilon) \mapsto \T \Star_\sigma \PPP(\varepsilon) \subset \PW \T \PPP_\infty$
gives $\PW \T \PPP_\infty$ the structure of a regular CW complex
whose cell poset is isomorphic to 
\[
	\{ (\sigma, \varepsilon) : \sigma \in \PPP, \varepsilon \in \Z_2(\sigma) \}.
\]
This is not quite the space we need: In analogy to the real-oriented blow-up
used in the classical case, we define 
the  \emph{tropical positive special fibre} $\PW \T \PPP_\infty^+$ as follows. 
Recall the poset 
\[
  Q(\PPP) = \{ (\sigma, \varepsilon) : \sigma \in \PPP, \varepsilon \in \Z_2(\RecCone(\sigma)) \}
\]
from \Cref{eq:Qposet}. 
Note that in comparison to the face poset of
$\PW \T \PPP_\infty$, in the definition of $Q(\PPP)$ we use $\RecCone(\sigma)$ instead of $\sigma$. 
The assignment $(\sigma, \varepsilon) \mapsto \T \Star_\sigma \PPP$ yields a directed system of topological spaces
and we define $\PW \T \PPP_\infty^+$ as the limit. 
By slight abuse of notation, we denote the copy of $\T \Star_\sigma \PPP$ labelled by $(\sigma, \varepsilon)$
as $\PW\T \Star_\sigma \PPP(\varepsilon)$.
We have a map $\PW \T \PPP_\infty^+ \to \PW \T \PPP_\infty$
which identifies strata labelled by $(\sigma, \varepsilon)$ and $(\sigma, \varepsilon')$ if $\varepsilon = \varepsilon'$ modulo
$T_{\Z_2}(\sigma)$. 

Now let $\X \subset \PPP$ be a subcomplex and $\EEE$ be a real phase structure for $\X$.
Each star $\Star_\tau \X$ carries an induced real phase structure $\Star_\tau \EEE$ given by quotienting
each $\EEE(\sigma)$ by $\T_{\Z_2}(\tau)$.  
We define the \emph{positive special fibre} $\PW(\overline{\X},\EEE)^+_\infty$ of $(\X, \EEE)$ as 
the polyhedral complex in $\PW \T \PPP_\infty^+$ whose intersection with
the stratum $\T \Star_\sigma \PPP(\varepsilon)$ labelled by $(\sigma, \varepsilon) \in Q(\PPP)$
is given by 
\[
  \PW(\overline{\X},\EEE)^+_\infty \cap \T \Star_\sigma \PPP(\varepsilon) = 
		\PW(\Star_\sigma \PPP, \Star_\sigma \EEE)(\pi(\varepsilon)).
\]
Here, $\pi \colon \Z_2(\RecCone(\sigma)) \to \Z_2(\sigma)$ denotes the canonical projection map. 
We set $\PW(\overline{X},\EEE)^+_\infty = |\PW(\overline{\X},\EEE)^+_\infty|$. 
We first state the fact that the pair 
$\PW(\overline{X},\EEE)^+_\infty \subset \PW \T \PPP^+_\infty$
is a regular CW pair and describe their poset structures.

\begin{corollary} \label{cor:SpecialFibreCWPair}
  The assignment $(\sigma, \varepsilon) \mapsto \T \Star_\sigma \PPP(\varepsilon)$
	gives $\PW \T \PPP^+_\infty$
	the structure of a regular CW complex whose cell poset is isomorphic to $Q(\PPP)$. 
	With respect to this CW structure, the pair 
	$\PW(\overline{X},\EEE)^+_\infty \subset \PW \T \PPP^+_\infty$
	is a regular CW pair such that the
	subset of cells intersecting $\PW(\overline{X},\EEE)^+_\infty$
	corresponds to $Q(\X, \EEE)$ from \Cref{eq:QposetII},
	\[
	  Q(\X, \EEE) = \{ (\sigma, \varepsilon) \in Q(\PPP) : \PW(\overline{X},\EEE)^+_\infty \cap \T \Star_\sigma \PPP(\varepsilon) \neq \emptyset\}. 
	\]
\end{corollary}

\begin{proof}
  This follows immediately from the construction using 
	\Cref{thm:StandardPairsTropical}.
\end{proof}

Finally, we have the following theorem relating tropical generic and special fibres.

\begin{thm} \label{thm:TropicalSide}
  The pairs $(\PW\T\Sigma, \PW(\overline{X}, \EEE))$ 
	and $(\PW \T \PPP^+_\infty, \PW(\overline{X},\EEE)^+_\infty)$
	are homeomorphic. 
	Moreover, 
	the homeomorphism can be chosen to respect the toric strata, that is, 
	such that	the cell $\T\Star_\rho \Sigma(\varepsilon)$
	is mapped to the union of cells $\T \Star_\sigma \PPP(\varepsilon)$
	with $\RecCone(\sigma) = \rho$. 
\end{thm}

\begin{proof}
  Fix $\varepsilon \in \Z_2^n$ and define
	$\PW \T \PPP^+_\infty(\varepsilon)$ as the union of the cells
	$\T \Star_\sigma \PPP(\pi(\varepsilon))$, $\sigma \in \PPP$, 
	$\pi \colon \Z_2^n \to \Z_2(\RecCone(\sigma))$.
	It is clear that $\PW \T \PPP^+_\infty(\varepsilon)$
	is a copy of $\T \PPP_\infty$ and $\PW\T\Sigma(\varepsilon)$
	is a copy of $\T\Sigma$.
	Moreover, by definition $\PW(\overline{X},\EEE)^+_\infty \cap 
	\PW \T \PPP^+_\infty(\varepsilon)$ is equal to $\PW(\overline{X},\EEE)(\varepsilon)_\infty$,
	since $\PW(\overline{\X},\EEE)(\varepsilon) \subset \X$ is the subcomplex consisting 
	of those faces
	$\sigma \in \X$ for which $\varepsilon \in \EEE(\sigma)$. 
	Hence by \Cref{prop:TropicalSpecialFibre} there is a CW homeomorphism
	from $\PW\T\Sigma(\varepsilon)$ to $\PW \T \PPP^+_\infty(\varepsilon)$
	which maps $\PW(\overline{X}, \EEE)(\varepsilon)$ to 
	$\PW(\overline{X},\EEE)^+_\infty \cap 
	\PW \T \PPP^+_\infty(\varepsilon)$.
	Moreover, it is clear that when running through $\varepsilon \in \Z_2^n$ 
	these homeomorphisms glue to a homeomorphism 
	from $\PW\T\Sigma$ to $\PW \T \PPP^+_\infty$ with the desired properties. 
\end{proof}

\begin{proof}[Proof of \Cref{UnimodularPatchworking}]
  By \Cref{thm:ClassicalSide}, \Cref{thm:TropicalSide} and 
	\Cref{cor:SpecialFibreCWPair} we know that both pairs
	$\R \overline{\XX}_t \subset \R\Sigma$ and 
	$\PW(\overline{X}, \EEE) \subset \PW\T\Sigma$ are regular CW pairs
	governed by the same pair of posets $Q(\X, \EEE) \subset Q(\PPP)$.
	We note that on both sides the torus orbit labelled by $\rho \in \Sigma$
	corresponds to the tuples $(\sigma, \varepsilon)$, $\RecCone(\sigma) = \rho$. 
	Hence by \Cref{prop:CWHomeomorphism},
	there exists an induced homeomorphism which respects 
	the stratification by torus orbits.  
\end{proof}

\subsection{Comparison to other real tropicalizations}
\label{sec:comparison}

  Let $\XX \subset (\C^*)^n \times \D^*$ be a real analytic family 
	with non-singular tropical
	limit $X = \Trop(\XX)$
	and associated real phase structure $\EEE = \EEE_\XX$.
	As noted before, we can consider $\XX$ as a variety over $\KK_\R$. 
	In \cite{JSY-RealTropicalizationAnalytification}, a \emph{real tropicalization}
	$\Trop_r(\XX) \subset \PW \R^n$ is defined along the following lines:
	Let $p$ be a $\bar{\KK}_\R$-point of $\XX$, where $\bar{\KK}_\R$ denotes the 
	field of real Puiseux series. 
	We have an associated 
	sign vector $\sign(p) \in \Z_2^n$ (the signs of the \enquote{leading coefficients}, written
	additively) 
	and the valuation vector $\val(p) \in \QQ^n$. 
	Then $\Trop_r(\XX)$ is defined by
	\[
		\Trop_r(\XX) \cap  \R^n(\epsilon) 
		:= \text{closure of } \{ -\val(p) : 
				 p \in \XX(\bar{\KK}_\R), \sign(p) = \epsilon \}.
	\]
	(In fact, in \cite{JSY-RealTropicalizationAnalytification}
	$\Trop_r(\XX)$ is defined as the exponential version of 
	this set in $(\R^*)^n$.)
  From the definition, it is immediately clear that 
	$\Trop_r(\XX) \subset \PW(X, \EEE)$:
	If $w = -\val(p)$, then the reduction $q = [p]$ of $p$ 
	is an $\R$-point of the initial variety $\init_w \XX$
	with $\sign(q) = \sign(p)$.
	In fact, we have $\PW(X, \EEE_\XX) = \Trop_r(\XX)$
	which follows from the following lemma
	(a real version of \cite[Lemma 2.1.3]{SpeyerThesis}).

\begin{lemma} 
  Let $\XX \subset (\bar{\KK}_\R^*)^n$ be a real analytic family 
	with non-singular tropical
	limit $X = \Trop(\XX)$. Fix $w \in X\cap \QQ^n$ and 
	$q \in \init_w \XX(\R)$. 
	Then there exists $p \in \XX(\bar{\KK}_\R)$ such that
	$w = -\val(p)$ and the reduction $[p]$ to $\init_w \XX(\R)$
	is $q$. 
\end{lemma}

Note that here, in slight conflict to the rest of the paper, 
$\XX$ and $\init_w \XX$ denote schemes over $\bar{\KK}_\R$
and $\R$, respectively. 

\begin{proof}
  After multiplication by $t^{-w}$ we may assume $w=0$. 
	We use the shorthand $\KK = \bar{\KK}_\R$ and 
	$R = \{a \in \KK : \val(a) \geq 0\}$. 
	Let $\XX_R$ denote the closure of $\XX$ in
	$T_R^n = \Spec R[x_1^\pm, \dots, x_n^\pm]$. 
	Note that $\XX_R$ is flat over $\Spec R$ and
	its only two fibers are $\XX_R \otimes \KK = \XX$
	and $\XX_R \otimes \R = \init_0 \XX$, 
	see \cite[Lemma 2.1.3]{SpeyerThesis}.
	Since $\XX$ is schön, see \Cref{prop:SchoenPropertiesI},
	$\XX$ and $\init_0 \XX$ are non-singular. 
	Hence, by \cite[\href{https://stacks.math.columbia.edu/tag/01V8}{Lemma 01V8}]{stacks-project}, 
	$\XX_R$ is non-singular over $\Spec R$. 
	Therefore there exists an open neighbourhood $q \in U \subset T_R^n$
	such that $U \cap \XX_R \subset U$ 
	is given by $n-d$ polynomials $f_1, \dots, f_{n-d} \in R[x_1^\pm, \dots, x_n^\pm]$ with
	full rank Jacobian on $U$. 
	Since $R$ is complete with respect to the valuation, by
	\cite[Proposition 4.11 (c)]{milneLEC}
	the solution $q$ of $[f_1], \dots, [f_{n-d}]$ can be lifted to a
	solution $p$ of $f_1, \dots, f_{n-d}$.
	Hence $p \in U(\R) \subset \XX(\bar{\KK}_\R)$ is the desired point. 
\end{proof}

\section{A spectral sequence for the homology of a patchwork}

In this section, given a non-singular tropical subvariety  $X$  of  a tropical toric variety $\T\Sigma$ and a real phase structure $ \mathcal{E}$ on $X$, we describe following \cite{RS} a filtration of $H_*(\PW X ; \Z_2)$ and we give a proof of Theorem \ref{thm:boundbetti}. The case of real phase structures on non-singular tropical hypersurfaces was treated in  \cite{RS}. 
We begin with an overview of tropical homology.

\subsection{Celluar cosheaves}
Let $\X$ be a polyhedral complex in $\T\Sigma$. A \emph{cellular cosheaf $\F$ on $\X$} is a collection of $\Z_2$-vector spaces $\F(\tau)$ for every $\tau \in \X$
together with linear maps $i_{\sigma \tau} \colon \F(\sigma) \to \F(\tau)$ for every inclusion of faces $\tau \subset \sigma$ such that
$i_{\sigma \sigma} = \text{id}$ and $i_{\tau \upsilon} \circ i_{\sigma \tau}$ for every chain $\upsilon \subset \tau \subset \sigma$.

Given a cellular cosheaf $\F$ on $\X$, the groups of \emph{cellular (Borel-Moore) $q$-chains with coefficients in $\F$} are 
\[
  C_q(\X; \F) = \bigoplus_{\substack{\dim \sigma = q \\ \sigma \text{ compact }} } \F(\sigma) \quad \text{ and } \quad C^{BM}_q(\X; \F) = \bigoplus_{\dim \sigma = q} \F(\sigma). 
\]
The boundary maps $\partial \colon C_q(\X ; \mathcal{F}) \to C_{q-1}(\X ; \mathcal{F})$ are 
the usual cellular boundary maps combined with the cosheaf maps $i_{\sigma \tau}$. 
Note that since we work with vector spaces over $\Z_2$, it is not necessary to choose orientations of the faces. 
The \emph{$q$-th (Borel-Moore) homology groups of $\F$ on $\X$} are
\[
  H_q(\X; \mathcal{F}) = H_q(C_{\bullet}(\X ;\mathcal{F})) \quad \text{ and } \quad H_q^{BM}(\X; \mathcal{F}) = H_q(C_{\bullet}^{BM}(\X ;\mathcal{F})).
\]

Let $X$ be a polyhedral space in $\T\Sigma$. A \emph{cellular cosheaf $\F$ on $X$} is a family of cellular cosheaves for any $\X$ such that $X = |\X|$
(by abuse of notation, we denote all these sheaves by $\F$) satisfying the following condition: If $\Y$ is a subdivision of $\X$, $\sigma \in \Y$,
and $\tau \in \X$ 
the unique face such that $\relint(\sigma) \subset \relint(\tau)$, then $\F(\sigma) = \F(\tau)$. 
In order to ensure that the homology groups are invariant under subdivision, we have to assume that our polyhedral complexes satisfy the following property. 

\begin{definition} \label{def:Xreg}
Let $\mathcal{X}$ be a polyhedral complex in $\T\Sigma$. We say that $\mathcal{X}$ is a regular polyhedral complex if the cellular structure induced by $\mathcal{X}$ on the one-point compactification $\hat{X}$ of $X$ is a regular CW-complex,
see \Cref{SecRegularCWComplexes}. 
\end{definition} 

For example, if $X=\R$ and $\mathcal{X}$ consists of a unique one dimensional cell, then it is not regular.
  From now on, we will only consider polyhedral spaces $X$ for which any
	subdivision $\X$ admits a regular refinement. 
	For such $X$, we define 
	\[
		H_q(X; \mathcal{F}) = H_q(\X; \mathcal{F}) \quad \text{ and } 
		\quad H_q^{BM}(X; \mathcal{F}) = H_q^{BM}(\X; \mathcal{F}),
	\]
	where $\X$ is some chosen regular subdivision of $X$. 
	The definition does not depend on this choice in the following sense:
	Any two regular subdivisions $\X$ and $\X'$ have a common regular refinement $\X''$. By 
	\cite[Theorem 7.3.9]{Curry}, the homology groups with respect to $\X$, $\X'$ and $\X''$ 
	are canonically isomorphic.

These homology groups of $X$ are all well-defined up to canonical isomorphisms. 
A \emph{morphism} $\F \to \G$ of cellular cosheaves on a polyhedral complex $\X$ is a collection of morphisms $\F(\tau) \to \G(\tau)$ for every $\tau \in \X$
which commute with the  restriction maps of both $\F$ and $\G$.  
We get the induced notions of \emph{kernels, cokernels, and exact sequences} of morphisms of cellular cosheaves on $\X$.

\subsection{Tropical homology} \label{sec:trophom}
Let $X$ be a non-singular tropical subvariety of a tropical toric variety $\T\Sigma$ defined by a fan $\Sigma$.
 If  $\T\Sigma_{\rho}$ and $\T\Sigma_{\eta}$ are a pair of strata corresponding to cones $\rho \subset \eta$ of $\Sigma$ then there is a canonical projection map of tangent spaces producing the map
\begin{equation}\label{eq:projmaps}
\pi_{\rho \eta} \colon T_{\Z_2} (\T\Sigma_{\rho})  \to T_{\Z_2} (\T\Sigma_{\eta}). 
\end{equation}
If $\sigma$ is a polyhedron of sedentarity $\rho$ such that $\sigma \cap \T\Sigma_{\eta} \neq \emptyset$, then
$\pi_{\rho \eta}$ maps $T_{\Z_2} (\sigma)$ onto $T_{\Z_2} (\sigma \cap \T\Sigma_{\eta})$.

\begin{definition}
Let $\X$ be a polyhedral complex in $\T\Sigma$.
The \emph{$p$-multi-tangent spaces} of $\X$ are cellular cosheaves $\F_p$ on $\X$ defined as follows. 
For a face $\tau \in \X$ of sedentarity $\rho$, we set 
\begin{equation}\label{def:Fp}
  \F_p(\tau) = \sum_{\substack{\tau \subset \sigma \subset \T\Sigma_{\rho}}} \bigwedge^p T_{\Z_2}(\sigma).
\end{equation}
Given $\tau' \subset \tau$ of sedentarity $\eta$, 
the maps of the cellular cosheaf 
$i_{ \tau \tau'} \colon \F_p(\tau) \to \F_p(\tau')$ 
are given by composing  $\wedge^p \pi_{\rho \eta}:  \bigwedge^p T_{\Z_2}(\sigma) \to \bigwedge^p T_{\Z_2}(\sigma \cap \T\Sigma_\eta)$ from (\ref{eq:projmaps})
with the inclusion maps 
$\bigwedge^p T_{\Z_2}(\sigma \cap \T\Sigma_\eta) \to \F_p(\tau')$.
\end{definition}

Clearly, the cosheaves $\F_p$ give rise to cellular cosheaves on polyhedral spaces $X$. 
The groups 
\[
  H_q(X; \mathcal{F}_p) \quad \text{ and } \quad H_q^{BM}(X; \mathcal{F}_p)
\]
are called the \emph{$(p,q)$-th (Borel-Moore) tropical homology group}.

\subsection{The sign cosheaf}\label{sec:cosheaf}

Let $\X$  be a polyhedral complex in $\T\Sigma$. Suppose $\X$ is  equipped with a real phase structure $\E$. 
We now recall the sign cosheaf defined in \cite{RS}.

For a face $\tau$ of $\X$ of sedentarity $\rho$, we set
\[
  \EEE(\tau) := \{ \varepsilon \in T_{\Z_2}(\T\Sigma_\rho) : \tau(\varepsilon) \in \PW(\X,\EEE)\}.
\]
Note that this coincides with the original $\EEE(\sigma)$ in the case when $\sigma$ is a facet. 
We define the abstract vector space $\mathcal{S}_{\mathcal{E}}(\tau)$ with basis given by the elements of $\mathcal{E}(\tau)$,
\begin{equation} \label{eq:SignCosheafSpaces} 
  \mathcal{S}_{\mathcal{E}}(\tau)=\Z_2\left\langle w_\varepsilon \mid \varepsilon\in\mathcal{E}(\tau)\right\rangle.
\end{equation}
The vector space $\mathcal{S}_{\mathcal{E}}(\tau)$ is a linear subspace of the abstract vector space 
$
\Z_2\left\langle w_\varepsilon\mid \varepsilon \in T_{\Z_2}(\T\Sigma_\rho)\right\rangle. 
$
If $\tau' \subset \tau$ is a face of sedentarity $\eta$, then $\tau(\varepsilon) \in \PW(\X, \EEE)$
implies that $\tau'(\pi_{\rho \eta}(\varepsilon)) \in \PW(\X, \EEE)$ (with the projection map
$\pi_{\sigma \tau}$ from (\ref{eq:projmaps}).
It follows that 
\begin{equation} \label{eq:SignCosheafMaps} 
  i_{\tau \tau'} \colon \S_\EEE(\tau) \to \colon \S_\EEE(\tau'), \quad  w_\varepsilon \mapsto w_{\pi_{\rho \eta}(\varepsilon)}
\end{equation}  
is a well-defined linear map for all $\tau' \subset \tau$.

\begin{definition}\label{def:signcosheaf}
The \emph{sign cosheaf} $\S_\mathcal{E}$ on $\X$ is the cellular cosheaf on $\X$ 
given by the spaces in (\ref{eq:SignCosheafSpaces}) and the cosheaf maps in (\ref{eq:SignCosheafMaps}).
It gives rise to a cellular sheaf $\S_\EEE$ on $X = |\X|$.
\end{definition}

The proof of the next proposition is exactly the same as the proof of Proposition 3.17 in \cite{RS}.
We denote by $C_{\bullet}(\PW(\X,\EEE); \Z_2)$ and $C^{BM}_{\bullet}(\PW(\X,\EEE); \Z_2)$ the ordinary
cellular chain complexes of $\PW(\X,\EEE)$ over $\Z_2$.

\begin{proposition}\label{prop:realhomcosheaf}
Let $X$ be a polyhedral subspace with real phase structure $\E$ and associated sign cosheaf $\S$. 
Let $\X$ be a polyhedral complex such that $X = |\X|$. 
Then we have 
isomorphisms of chain complexes
$$C_{\bullet}(\X; \S) \cong C_{\bullet}(\PW(\X,\EEE); \Z_2) \quad \text{and} \quad  C^{BM}_{\bullet}(\X; \S) \cong C^{BM}_{\bullet}(\PW(\X,\EEE); \Z_2)$$
and therefore also isomorphisms 
$$H_{q}(X; \S) \cong H_{q}(\PW(\X,\EEE); \Z_2) \quad \text{and} \quad H^{BM}_{q}(X; \S)  \cong H^{BM}_{q}(\PW(\X,\EEE); \Z_2).$$
\end{proposition}

\subsection{The filtration of the sign cosheaf}
We describe here a filtration of the sign cosheaf $\S_\mathcal{E}$. This filtration is a direct generalisation of the one described in \cite{RS} in the case of hypersurfaces. 
Given a subset $H \subset T_{\Z_2}(\rho)$, we denote by $\Aff_p(H)$ the set of all $p$-dimensional affine subspaces of $T_{\Z_2}(\rho)$ contained in $H$.
For a subset $G \subset T_{\Z_2}(\rho)$, set $$ w_{G} := \sum_{\epsilon \in G} w_{\epsilon} \in \Z_2\left\langle w_\varepsilon \mid \varepsilon\in T_{\Z_2}(\rho) \right\rangle.$$  
Finally, we define 
\[
  K_p(H) := \langle  w_{G} \ | \ G \in \Aff_p(H)   \rangle.
\]

\begin{definition}
\label{def:Kp}
Let $\X$ be a polyhedral complex of sedentarity $\rho$ in $\T\Sigma$ with real phase structure $\mathcal{E}$.
For all $p$, we  define a collection of cosheaves $\K_p$ on $\X$. 
On a face $\tau$ of $\X$ of sedentarity $\eta$, we set  
\begin{equation} \label{eq:DefKp} 
  \K_p(\tau) = \sum_{\sigma \supset \tau} K_p(\pi_{\rho \eta} (\EEE(\sigma))) \subset \S_\mathcal{E}(\tau),
\end{equation}
where the sum runs through facets of $\X$. 
If $\tau' \subset \tau$, the cosheaf maps $\K_{p}(\tau) \to \K_{p}(\tau')$ 
are the restrictions of the maps $i_{\tau \tau'} \colon \S_{\mathcal{E}}
(\tau) \to \S_{\mathcal{E}}(\tau'). $
The cosheaves $\K_p$ give rise to cosheaves on $X = |\X|$ which we denote by the same letter.
\end{definition}

\begin{remark}
Note that, with the assumption from above,
\[
  \EEE(\tau) = \bigcup_{\sigma \supset \tau} \pi_{\rho \eta} (\EEE(\sigma)),
\]
where the union runs through facets of $\X$. This explains the inclusion $\K_p \subset \S_\EEE$ mentioned in the 
above definition. In fact, the cosheaves $\K_p$ form a filtration of $\S_\EEE$ given by
\begin{equation}\label{filttau} \K_{d} \subset \dots \subset  \K_{2} \subset \K_{1} \subset \K_{0} =\S_\mathcal{E} .
\end{equation}
Indeed, any affine subspace $G$ of dimension $p$ can be written as disjoint union $G = H_1 \sqcup H_2$ of two
affine subpaces of dimension $p-1$. Then $w_G=w_{H_1}+w_{H_2}$, which shows $\K_p \subset \K_{p-1}$.
\end{remark}

  Assume that $H \subset T_{\Z_2}(\T\Sigma_\rho)$ is a affine subspace. 
	Then by general linear algebra we have 
	\begin{equation} \label{eq:wedgemaps} 
		K_p(H) / K_{p+1}(H) \cong \bigwedge^p T(H),
	\end{equation}
	induced by mapping $w_{G}$ to $v_1 \wedge \dots \wedge v_p$. 
	Here $G \in \Aff_p(M)$ and $v_1, \dots, v_p$ form a basis of $T(G)$. 
	Furthermore, note that $\F_p$ can be written in the form
\[
  \F_p(\tau) = \sum_{\sigma \supset \tau} \bigwedge^p \pi_{\rho \eta} (T_{\Z_2}(\sigma))
\]
analogous to (\ref{eq:DefKp}). 
Since 
\[
  \pi_{\rho \eta} (T_{\Z_2}(\sigma)) = \pi_{\rho \eta} (T_{\Z_2}(\EEE(\sigma))) = T_{\Z_2}(\pi_{\rho \eta} (\EEE(\sigma))),
\]
the restriction of the map $\K_p(T_{\Z_2}(\T\Sigma_\rho)) \to \bigwedge^p T_{\Z_2}(\T\Sigma_\rho)$ from (\ref{eq:wedgemaps})
to $\K_p(\tau)$ takes image in $\F_p(\tau)$. Hence, it gives rise to a morphism of cellular cosheaves $\K_p \to \F_p$.
Applying (\ref{eq:wedgemaps}) to each summand, we see that this morphism is surjective and $\K_{p+1}$ lies in the kernel.
Since the sum is not a direct sum, in general, we cannot conclude that $\K_{p+1}$ is \emph{equal} to the kernel. 
Equality holds, however, if $\X$ represents a non-singular tropical variety. This is the content of the following proposition.

\begin{proposition} \label{prop:exactcosheaf}
Let $X$ be a non-singular tropical variety in $\T\Sigma$ with a real phase structure $\EEE$. 
Then for all $p$ there is an exact sequence of cosheaves
$$ 0 \to \K_{p+1} \to \K_p \to \F_p \to 0.$$
\end{proposition}

\begin{proof}
The proof follows the same lines as the proof of Proposition 4.10 in \cite{RS}.
As mentioned before, it is true in general that $\K_{p+1} \to \K_p$ (the inclusion map) is injective,
$\K_p \to \F_p$ is surjective, and $\K_{p+1} \subset \ker(\K_p \to \F_p)$. 
We prove equality in the last statement by a comparison of dimensions. 
More precisely, let $\X$ be a polyhedral complex such that $X = |\X|$ and pick $\tau \in \X$. 
It suffices to show that 
$$
\sum_{ p = 0}^{\dim X} \dim \F_p(\tau)=\dim\S_\E(\tau).
$$
This statement is a local statement and so by 
Definition \ref{def:nonsingular} we may assume $\X = \overline{\projFan}_M$ for some matroid $M$. 
By \cite[Theorem 1.1]{RRSmat}, there is an oriented matroid $\MMM$ representing the real phase structure
on $\X$.

The dimension of $S_\E(\tau)$ is equal to $|\cup_{\sigma \in \mathcal{X}} \mathcal{E}(\sigma)|$ and so $2 \dim S_\E(\tau)$ is equal to the number of topes of the oriented matroid of $M$.    By
 \cite{Zaslavsky},  the number of topes of $\mathcal{M}$ is equal to $(-1)^r \chi_M(-1)$, where $\chi_M$ is its characteristic polynomial of the underlying matroid $M$ of $\mathcal{M}$, see also 
 \cite[Theorem 4.6.1]{bjorner}. Thus  $\dim S_\E(\tau) = (-1)^r \tilde{\chi}_M(-1)$ where $\tilde{\chi}_M(t)$ is the reduced characteristic polynomial of $M$. 
Finally, by   \cite{ZharkovOS},  we obtain 
$$
(-1)^r \tilde{\chi}_M(-1)=\sum_{0\leq p \leq r} \dim \F_p(\Sigma_M),
$$
and the proof follows.
\end{proof}

\subsection{The spectral sequence}

We now establish the spectral sequence relating the homology of the patchwork of a tropical manifold to the tropical homology groups. 

We denote the complex of relative chains by
\[
  C_{\bullet}^\diamond(\X ; \K_{p}, \K_{p+1}):= C_{\bullet}^\diamond(\X; \K_{p})/C_{\bullet}^\diamond(\X; \K_{p+1})
\]
where the $\diamond$ denotes either normal or Borel-Moore homology.
This is the $0$-th page of the spectral sequence for the chain complex $C^{\diamond}_{\bullet}(\X ;\S_{\mathcal{E}})$ obtained by the filtration by the $\K_p$ cellular cosheaves.  
Thanks to Proposition \ref{prop:exactcosheaf}, we obtain the following statement,
analogous to \cite[Proposition 4.12]{RS}.

\begin{cor}
\label{prop:firstpage}
Let $X$ be a non-singular tropical variety in $\T\Sigma$ with a real phase structure $\EEE$, and 
suppose $\X$ be a polyhedral complex such that $X = |\X|$.
The first page of the spectral sequence associated to the filtration of the chain complex  $C_{\bullet}^\diamond(\X ;\S_{\mathcal{E}})$ by  the chain complexes $C_{\bullet}^\diamond(\X ;\K_p)$ has  terms 
$$E^{1,\diamond}_{q,p} \cong  H_q^\diamond(X; \F_p).$$
\end{cor}

\begin{proof}[Proof of Theorem \ref{thm:boundbetti}]
	The pages of a  spectral sequence satisfy  
	$\dim E^{\infty}_{q,p} \leq \dim E^r_{q,p}$ for all $r$.  
	By \Cref{prop:firstpage}, 
	and the convergence of the spectral sequence associated to the filtration  we obtain 
	\[ 
		\dim H_q^\diamond(X; \S_{\mathcal{E}})  = 
		\sum_{p = 0}^{\dim X}  E^{\infty,\diamond}_{q,p}  \leq  
		\sum_{p = 0}^{\dim X}  E^{1,\diamond}_{ q,p} =  
		\sum_{p = 0}^{\dim X} H_q^\diamond(X; \mathcal{F}_p).
	\]
	The first part of the  theorem now follows since 
	$\dim H_q^{\diamond} (\PW(\X,\EEE); \Z_2) = \dim H_q^\diamond(X; \S_{\mathcal{E}}) $ by \Cref{prop:realhomcosheaf}. 
	
For the relation between the tropical signature and the Euler characteristic of $\PW(\X,\EEE)$, using again 
\Cref{prop:realhomcosheaf} we have  $\chi^\diamond(\PW(\X,\EEE)) = \chi^\diamond(\S_{\EEE})$, and moreover,  
		\[
	  \chi^\diamond(\S_{\EEE}) = \sum_{p,q} (-1)^q E^{\infty,\diamond}_{q,p}
		  = \sum_{p,q} (-1)^q E^{1,\diamond}_{q,p} = \sum_p \chi^\diamond(\F_p) = \sigma^\diamond(X).
	\]
This completes the proof. 
\end{proof}

\subsection{Hirzebruch polynomials} 

We conclude by proving that in the context of \Cref{sec:patchworkinglimit}, the Hirzebruch genus 
can be computed from the tropical limit. 
Recall that the Hirzebruch genus of $\overline{\XX}_t$ is 
\begin{align} \label{eq:classicalHirzebruch} 
	\chi_y(\overline{\XX}_t)  = \sum_{p,q} e_c^{p, q}(\overline{\XX}_t)y^p,
\end{align}
where $e_c^{p, q}(\overline{\XX}_t)=\sum_k (-1)^k h^{p,q}(H_c^k (\XX_t))$ 
is defined as in the introduction. 
 
Given a tropical non-singular subvariety $X \subset \T\Sigma$
of dimension $d$, we define
its \emph{tropical Hirzebruch genus} as 
\[
  \chi_y(X) := \sum_{p, q=0}^d (-1)^{p+q} \dim H_{q}^\BM (X, \F_p) y^p.
\]
For the reader familiar with the cohomology versions of tropical $(p,q)$-groups, we note that
(using the universal coefficient theorem for the field $\Z/2\Z$) 
\[
\chi_y(X) = \sum_{p=0}^d (-1)^{p+q} \dim H_c^q(\X; \F^p) y^p,
\]
in accordance with the classical definition \eqref{eq:classicalHirzebruch}.

\begin{prop}\label{prop:tropicalepoly}
	Let $\XX$ be a analytic family over $\D^*$
	with tropical limit $X = \Trop(\XX)$. 
	Let $\Sigma$ be a pointed unimodular fan 
	and consider the closures $\overline{\XX} \subset \C\Sigma \times \D^*$ and 
	$\overline{X} \subset \T\Sigma$. Assume that $\overline{X}$ is non-singular. 
	Then for generic
	$t \in \D^*$,
	we have
	\[
	  \chi_y(\overline{\XX}_t) = \chi_y(\overline{X}).
	\]
\end{prop}

If $\C\Sigma = \CP^n$, the statement follows from the
stronger equality of classical and tropical Hodge numbers
proven in \cite{IKMZ}. Proposition \ref{prop:tropicalepoly} was already proven in the case of hypersurfaces, for any $\Sigma$, in \cite{ARS}.

\begin{proof} 
  Given $\rho \in \Sigma$, we set $\XX^\rho := \overline{\XX} \cap (\C\OOO_\rho \times \D^*)$
	and $X^\rho := \overline{X} \cap \T\OOO_\rho$. Note that $\Trop(\XX^\rho) = X^\rho$.
	We first claim that both sides of the equation are motivic in the sense
	that
	\begin{align*} 
		\chi_y(\overline{\XX}_t) &= \sum_{\rho \in \Sigma} \chi_y(\XX^\rho_t), &
		\chi_y(X)                  &= \sum_{\rho \in \Sigma} \chi_y(X^\rho).
	\end{align*}
	The first equation follows from the fact that the Hirzebruch genus is motivic.
	On the tropical side, it is just a simple computation. 
	If $\overline{\X}$ is a regular subdivision
	of $\overline{X}$ we get
	\begin{align*} 
		\chi_y(\overline{X}) &= \sum_{p,q} (-1)^{p+q} \dim H_{q}^\BM (X, \F_p) y^{p} \\
			&= \sum_p \sum_{\substack{\sigma \in \overline{\X}}} (-1)^{p+\dim \sigma} \dim \F_p(\sigma) y^{p}\\
			&= \sum_{\rho \in \Sigma} 
			   \sum_p \sum_{\text{Sed} \sigma = \rho} (-1)^{p+\dim \sigma} \dim \F_p(\sigma) y^{p}\\
			&= \sum_{\rho \in \Sigma}  \chi_y(X^\rho).
	\end{align*}
	It hence suffices to consider the case $\Sigma = \{0\}$.
	By Lemma \ref{lem:Refinements}, we can choose a strictly unimodular subdivision $\PPP$ of
	$\R^n$ such that $\X = \{\sigma \in \PPP : \sigma \subset X\}$
	is a subdivision of $X$.
	We can now apply \cite[Corollary 1.4]{KS-TropicalGeometryMotivic2}, which proves that for generic $t$
	\[
	  \chi_y(\XX_t) 
		  = \sum_{\substack{\sigma \in \X \\ \sigma \text{ bounded}}} 
			  (-1)^{\dim \sigma}\chi_y(\init_\sigma \XX_t).
	\]
	Again, we show that the analogous statement holds tropically.
	To do so, we denote by $S(\sigma)$ the \emph{non-reduced} star of 
	$X$ at $\sigma$. That is, $S(\sigma)$ is the $d$-dimensional fan in $\R^n$
	such that $S(\sigma) /T(\sigma) = \Star_\sigma X$. Applying Poincaré duality \cite{JRS-Lefschetz11Theorem} we obtain that
	\begin{align*} 
		\chi_y(X) &= \sum_{p,q} (-1)^{p+q} \dim H_{q}^\BM (X, \F_p) y^p \\
			&= \sum_{p,q} (-1)^{p+q} \dim H_{d-q} (X, \F_{d-p}) y^p \\
			&= \sum_p \sum_{\substack{\sigma \in \X \\ \sigma \text{ bounded}}}
			   (-1)^{p+d-\dim \sigma} \dim \F_{d-p}(\sigma) y^{p}\\
			&= \sum_{\substack{\sigma \in \X \\ \sigma \text{ bounded}}} (-1)^{\dim\sigma} \sum_p
			   (-1)^{p+d} \dim H_{0} (S(\sigma), \F_{d-p}) y^p \\
			 &= \sum_{\substack{\sigma \in \X \\ \sigma \text{ bounded}}} (-1)^{\dim\sigma} \sum_p
			   (-1)^{p+d} \dim H_{d}^\BM (S(\sigma), \F_{p}) y^p \\  
			&= \sum_{\substack{\sigma \in \X \\ \sigma \text{ bounded}}} (-1)^{\dim\sigma}
			   \chi_y(S(\sigma)). 
	\end{align*}
	Note that $\Trop(\init_\sigma \XX) = S(\sigma)$, so we have reduced to the case
	where $\XX=\bL$ is linear. Let $M = M_\bL$ be the associated matroid and let
	$\overline{\mu}_M(y)$ denote its reduced characteristic polynomial. By \cite[Lemma 7.5]{KS-TropicalGeometryMotivic}, the Hirzebruch genus of $\bL$ is equal to the reduced characteristic polynomial of $M$. Moreover, by  \cite{ZharkovOS} and \cite{OrlikTerao}, the dimension of $\F_p(\projFan_M)$ is the $(d-p)$-th coefficient of $\overline{\mu}_M(y)$. 
	Then we get
	\begin{align*} 
		 \chi_y(\bL) &= \overline{\mu}_M(y) \\
			&= \sum_p (-1)^{d-p} \dim \F_{d-p}(\projFan_M) y^p \\
			&= \sum_p (-1)^{p+d} \dim H_{d}^\BM (\projFan_M,\F_{p}) y^p \\
			&= \chi_y(\projFan_M). 
	\end{align*}
	This proves the claim. 
\end{proof}

\begin{proof}[Proof of \Cref{cor:eulerSignature}]
Evaluating the polynomials from \Cref{prop:tropicalepoly} in $y = -1$, one obtains
$$
\sigma_c(\overline{\XX}_t) = \sigma^\BM(\overline{X}).
$$
The latter is equal to $\chi^\BM(\RR \overline{\XX}_t )$ by combining Theorems \ref{Patchworking} and \ref{thm:boundbetti}. 
\end{proof}

\bibliographystyle{alpha}
\bibliography{biblio}
\end{document}